\documentclass[leqno]{article}
\usepackage{amsmath,amsfonts,amsthm,amssymb,indentfirst}

\newcommand{\<}{\kern.0833em}

\newtheorem{theorem}{Theorem}
\newtheorem{lemma}[theorem]{Lemma}
\newtheorem{corollary}[theorem]{Corollary}
\newtheorem{proposition}[theorem]{Proposition}
\newtheorem{definition}[theorem]{Definition}
\newtheorem{question}[theorem]{Question}

\newcommand{\fb}{\mathbf}
\newcommand{\R}{\mathbb R}
\newcommand{\E}{\mathbb E}
\newcommand{\M}{\mathbf{Metr}}
\newcommand{\rad}{\mathrm{rad}}
\newcommand{\mr}{\mathrm{map\mbox{-}rad}}
\newcommand{\diam}{\mathrm{diam}}
\makeatletter
\newcommand{\xlabel}{\stepcounter{equation}
  \gdef\@currentlabel{\p@equation\theequation}{\rm(\@currentlabel)}}
\makeatother
\newenvironment{xlist}
  {\begin{list}{\xlabel}
    {\setlength{\rightmargin}{20pt}
     \setlength{\leftmargin}{37pt}
     \setlength{\labelsep}{20pt}
     \setlength{\labelwidth}{20pt}}}
  {\end{list}}

\begin{document}

\title{Mapping radii of metric spaces%
\thanks{
%  	 special maps on m' sps
Secondary: 46B20, 46E15, 52A40.
%  geom of n'l/sps,B sps of fns, cvxty extremum problems
\protect\\\indent
Keywords: nonexpansive map between metric spaces,
maximum radius of image, convex subset of a normed vector space.
\protect\\\indent
Any updates, errata, related references etc., learned of
after publication
will be noted at http://math.berkeley.edu/\protect\linebreak[0]%
{$\!\sim$}gbergman\protect\linebreak[0]/papers/\,.
}}
\author{George M. Bergman}
\maketitle

\begin{center}{\em Dedicated to the memory of David Gale}\end{center}

\begin{abstract}
It is known that every closed curve of length $\leq 4$ in $\R^n$ $(n>0)$
can be surrounded by a sphere of radius~$1,$ and that this is the best
bound.
Letting $S$ denote the circle of circumference $4,$ with the
arc-length metric, we here express this fact by saying that the
{\em mapping radius} of $S$ in $\R^n$ is~$1.$

Tools are developed for estimating the mapping radius of
a metric space $X$ in a metric space $Y.$
In particular, it is shown that for $X$ a bounded metric space, the
supremum of the mapping radii of $X$ in all convex subsets of normed
metric spaces is equal to the infimum of the $\sup$ norms of all
convex linear combinations of the functions $d(x,-):X\rightarrow\R$
$(x\in X).$

Several explicit mapping radii are calculated, and
open questions noted.
\end{abstract}
% - - - - - - - - - - - - - - - - - - - - - - - - - - - - - -

\section{The definition, and three examples.}\label{S.intro}

\begin{definition}\label{D.imr}
We will denote by $\M$ the category whose objects are
metric spaces, and whose morphisms are nonexpansive maps.
That is, for metric spaces $X$ and $Y$ we let
\begin{xlist}\item\label{x.MXY}
$\M(X,Y)\ =\ \{f:X\rightarrow Y\mid(\forall\,x_0,\<x_1\in X)\ \,
d(f(x_0),f(x_1))\leq d(x_0,x_1)\}.$
\end{xlist}
Throughout this note, a {\em map} of metric spaces will mean
a morphism in $\M.$

Given a nonempty subset $A$ of a metric space $Y,$ we define its
{\em radius} by
\begin{xlist}\item\label{x.rad}
$\rad_Y(A)\ =\ \inf_{y\in Y}\ \sup_{a\in A}\ d(a,y),$
\end{xlist}
a nonnegative real number or $+\infty.$
For metric spaces $X$ and $Y,$ we define the {\em mapping radius}
of $X$ in $Y$ by
\begin{xlist}\item\label{x.ir}
$\mr(X,Y)\ =\ \sup_{f\in\M(X,Y)}\ \rad_Y(f(X))\ \\[6pt]
\hspace*{.98in}
=\ \sup_{f\in\M(X,Y)}\ \inf_{y\in Y}\ \sup_{x\in X}\ d(f(x),y).$
\end{xlist}
If $X$ is a metric space and $\fb{Y}$ a class of metric spaces, we
likewise define
\begin{xlist}\item\label{x.irfb}
$\mr(X,\fb{Y})\ =\ \sup_{Y\in \fb{Y}}\ \mr(X,Y)\\[6pt]
\hspace*{.99in}
=\ \sup_{Y\in \fb{Y},\,f\in\M(X,Y)}\ \inf_{y\in Y}\ %
\sup_{x\in X}\ d(f(x),y).$
\end{xlist}
\textup{(}The term ``mapping radius'' occurs occasionally in complex
analysis with an unrelated meaning \textup{\cite[Def.\,7.11]{EH}.)}

All {\em vector spaces} in this note will be over the field of real
numbers unless the contrary is stated.
\end{definition}

The result stated in the first sentence of the
abstract has been discovered many times
\cite{C+K}, \cite{Chern}, \cite{N}, \cite{R+HS}, \cite{JW}.
(Usually, the length of the closed curve is given as $1$
and the radius of the sphere as $1/4,$ but
the scaled-up version will be more convenient here.)
Let us obtain it in somewhat greater generality.

\begin{lemma}\label{L.S}
Let $S$ denote the circle of circumference $4,$ with the arc-length
metric.
Then for any nonzero normed vector space $V,$ we have $\mr(S,V)=1.$
\end{lemma}\begin{proof}
In $V,$ any $\!1\!$-dimensional subspace $U$ is isometric to the real
line $\R,$ and we can map $S$ into $\R,$ and hence into $U,$
by ``folding it flat'', getting for image an interval of length~$2.$
Since this interval has points at distance~$2$ apart,
its radius in $V$ cannot be less than $1,$ so $\mr(S,V)\geq 1.$

For the reverse inequality, consider any map $f: S\rightarrow V.$
We wish to find a point $y\in V$ having
distance $\leq 1$ from every point of $f(S).$
Let $p$ and $q$ be any two antipodal points of
$S,$ and let
\begin{xlist}\item[]
$y=(f(p)+f(q))/2.$
\end{xlist}
Every point $x\in S$ lies on a length-$\!2\!$ arc between
$p$ and $q$ in $S,$ hence $d(p,x)+d(q,x)=2,$ hence
$d(f(p),f(x))+d(f(q),f(x))\leq 2,$ i.e.,
$(d(f(p),f(x))+d(f(q),f(x)))/2\leq 1,$ so
$d((f(p)+f(q))/2,f(x))\leq 1,$ as claimed.
\end{proof}

Let us make explicit the argument used at the very last step above.
It is the $c_1=c_2=1/2$ case of

\begin{xlist}\item\label{x.ci}
If $c_1,\dots,c_n$ are nonnegative real numbers summing to $1,$
and $v_1,\dots,v_n$ are elements of a normed vector space $V,$
then for all $w\in V,$
$d(\sum\,c_i\,v_i,\,w)\,\leq\,\sum\,c_i\,d(v_i,\,w).$
\end{xlist}
This can be seen by writing the left-hand side as
$||(\sum\,c_i\<v_i)-w||=||\sum\,c_i\<(v_i-w)||\leq
\sum\,||c_i\<(v_i-w)||=\sum\,c_i\,d(v_i,\,w).$
\vspace{6pt}

Consider next the union $X$ of two circles $S_0$
and $S_1,$ each of circumference~$4,$ intersecting in a pair of
points antipodal in each (e.g., take for $S_0$ and
$S_1$ any two distinct great circles on
a sphere of circumference~$4),$ again with the arc-length metric.
We can show that this $X$ also has mapping radius $\leq 1$ in $V$ by the
same argument as before, except that where we previously
used an arbitrary pair of antipodal points, we are now
forced to use precisely the pair at which our circles intersect.
We are not so restricted in the example showing that radius~$1$ can
actually be achieved -- we can stretch one circle taut between any two
antipodal points, and for most choices of those points, we have a
great deal of freedom as to what to do with the other circle.
In any case, we have

\begin{lemma}\label{L.2S}
Let $X$ be the union of two circles $S_0$ and $S_1,$ each of
circumference~$4,$ intersecting in a pair of points antipodal in each,
with the arc-length metric.
Then for any nonzero normed vector space $V,$ we have $\mr(X,V)=1.$\qed
\end{lemma}

We could apply the same method to any number of circles
joined at a common pair of antipodal points; but let us move in a
different direction.
Again picturing $S_0$ and $S_1$ as great circles on a sphere of
circumference~$4$ in Euclidean $\!3\!$-space, assume they meet at
right angles, and call
their points of intersection the north and south poles.
Let us bring in a third circle, $S_2,$ the equator, and let
$X=S_0\cup S_1\cup S_2,$ again with the arc length metric.

We no longer have a pair of antipodal points belonging to all three
circles; rather, we have three pairs of points,
$S_1\cap S_2=\{p_0,q_0\},$
$S_2\cap S_0=\{p_1,q_1\},$ and $S_0\cap S_1=\{p_2,q_2\}.$
Now given a normed vector space $V$ and
a map $f:X\rightarrow V$ in $\M,$ suppose we let
\begin{xlist}\item\label{x.sum/6}
$y\ =\ (f(p_0)+f(q_0)+f(p_1)+f(q_1)+f(p_2)+f(q_2))\</\<6.$
\end{xlist}
What can we conclude about $d(y,f(x))$ for $x\in X$?

Say $x\in S_2.$
Since both $\{p_0,q_0\},$ and $\{p_1,q_1\}$ are pairs of
antipodal points of $S_2,$ we have
$d(p_0,x)+d(q_0,x)=d(p_1,x)+d(q_1,x)=2.$
The same will not be true of $d(p_2,x)$ and $d(q_2,x).$
To determine how large these can get, let us take
$x\in S_2$ as far as possible (under our arclength metric)
from the intersections of $S_2$ with our
two circles through the poles $p_2$ and $q_2.$
This happens when $x$ is at the
midpoint of any of the quadrants into which $p_0,$ $q_0,$ $p_1$ and
$q_1$ divide $S_2;$ in this situation, $d(p_2,x)=d(q_2,x)=3/2.$
(Each quadrant has arc-length~$1,$ and one has to go a quadrant
and a half to get from $p_2$ or $q_2$ to $x.)$
We see, in fact, that for any $x\in S_2$ we have
$d(p_2,x)=d(q_2,x)\leq 3/2,$ hence $d(p_2,x)+d(q_2,x)\leq 3.$
Now applying any map $f:X\rightarrow V,$
and invoking~(\ref{x.ci}) with all $c_i=1/6,$ we see that
for $y$ as in~(\ref{x.sum/6}) we have $d(y,f(x))\leq (2+2+3)/6 = 7/6.$
We have proved this for $x\in S_2;$ by symmetry, it is also true
for $x$ lying on $S_0$ or $S_1.$
This allows us to conclude,
not that $\mr(X,V)=1$ as in the preceding two cases, but that
\begin{xlist}\item\label{x.7/6}
$\mr(X,V)\ \leq\ 7/6.$
\end{xlist}

And in fact, there do exist maps $f:X\rightarrow V$ with
$\rad_V(f(X))>1.$
To describe such a map, note that $X$ can be identified with
the $\!1\!$-skeleton of a regular octahedron of edge~$1.$
In the next few paragraphs, let us put aside our picture of
$X$ in terms of great circles on a
sphere, and replace it with this (straight-edged) octahedral skeleton.

If we look at our octahedron
in Euclidean $\!3\!$-space from a direction perpendicular to one
of its faces, we see that face and the opposite one as overlapping,
oppositely oriented equilateral
triangles, with vertices joined by the remaining 6 edges, which
look like a regular hexagon.
Now suppose we regard these two opposite triangular
faces as made of stiff wire, and the other 6 edges as made of string.
Then if we bring the planes of the two
wire triangles closer to one another, the string edges will loosen.
Suppose, however, that we rotate the top triangle
clockwise as they approach one another, so that three of those
strings are kept taut, while the other three become still looser.
When the planes of our wire triangles meet, those wire triangles
will coincide, and the three taut string edges will fall together with
the three edges of that triangle, while the three loose
ones become loops, hanging from the three vertices.
Let us lock the two wire triangles together,
and pull the three loops taut, radially
away from the center of symmetry of the triangle.

What we then have is the image of a certain map $f$ in $\M$
from our octahedral skeleton $X$ into a plane, which we can
identify with $\R^2.$
We see that $\rad_{\R^2}(f(X))$ will be the distance from the center of
symmetry of our
figure to each of the three points to which the drawn-out loops are
stretched; i.e., the sum of the distance from the center of
symmetry to each vertex of the triangle, and the length
of the stretched loop attached thereto.
The former distance is two thirds of the altitude of the triangle,
$(2/3)(\sqrt{3}/2)=1/\sqrt{3}\,,$ and the latter length is $1/2$
(since the loop doubles back), so
\begin{xlist}\item\label{x.1/2+1/sqrt3}
$\rad_{\R^2}(f(X))\ =\ 1/\sqrt{3}+1/2.$\quad
Hence, $\mr(X,\R^2)\ \geq\ 1/\sqrt{3}+1/2\ >\ 1.$
\end{xlist}

This shows that our three-circle space does indeed behave
differently from the preceding one- and two-circle examples.

However $1/\sqrt{3}+1/2\approx 1.0773,$ which falls well short
of the upper bound $7/6\approx 1.1667$ of~(\ref{x.7/6}).

We can overcome this deficiency by using a different norm on $\R^2.$
Let $V$ be $\R^2$ with the norm whose unit disc is the
region enclosed by a regular hexagon $H$ of unit side.
Note that the 6 sides of $H$ are parallel to the 6 radii joining
$0$ to the vertices of $H,$ hence these sides have length~$1$ in the new
metric, just as in the Euclidean metric, and indeed, any line segment
in one of those directions will have the same length in both metrics.
Now let us map $X,$ still pictured as the $\!1\!$-skeleton
of a regular octahedron of side~$1$ in Euclidean $\!3\!$-space, into
$V$ so that, as before, two opposite triangles are embedded
isometrically (now under the metric of $V),$ and made to fall together
with each other and with three of the other edges, while the
remaining three edges form loops that are stretched radially outward
as far as they will go.
Let us moreover take the sides
of our image-triangle to be parallel to three sides of~$H.$

The map $X\rightarrow \R^2$ that does this is almost the same
one as before.
The 9 edges that end up parallel to edges of $H$ are mapped
exactly as before, since distances in those directions
are the same in the two metrics.
The three folded loops end up set-theoretically
smaller than before,
since the new metric is greater in their direction than is the Euclidean
metric, and they go out a distance $1/2$ in the {\em new} metric
before turning back; but they still contribute the value $1/2$ to the
calculation of the radius of our image of $X.$
The significant change in that calculation concerns the distance
from the center of our triangle to its three vertices.
Looking at our triangle as a translate of one of the 6 equilateral
triangles into which $H$ is decomposed by its radii,
we see that the altitude of that triangle is
equal to its side in this metric (since the midpoint of a side of $H$
has the same distance,~$1,$ from the origin as a vertex of $H$ does).
Hence the distance from the center to a vertex is $2/3.$
Adding to this the distance $1/2$ from that vertex to
the end of the loop attached to it, we get $2/3+1/2=7/6.$
Assuming that the center of our triangle is indeed the
minimizing point defining the radius (i.e., is a value of $y$ that
yields the infimum~(\ref{x.rad}); we will verify this
in Lemma~\ref{L.sym}), this achieves the upper bound~(\ref{x.7/6}).
Summarizing, and making a few supplementary observations, we have

\begin{lemma}\label{L.3S}
Let $X$ be the $\!1\!$-skeleton of a regular octahedron of side~$1,$
under the arc-length metric.
Then for any nonzero normed vector space $V,$
\begin{xlist}\item\label{x.1<7/6}
$1\ \leq\ \mr(X,V)\ \leq\ 7/6.$
\end{xlist}

The exact value of $\mr(X,V)$ is~$1$ if $V$ is $\!1\!$-dimensional,
is $\geq 1/\sqrt{3}+1/2$ if $V$ is $\R^n$ $(n\geq 2)$ under
the Euclidean norm, and is $7/6$ if $V$ is $\R^2$ under the norm
having for unit circle a regular hexagon.
\end{lemma}\begin{proof}
The lower bound~$1$ in~(\ref{x.1<7/6}) is gotten as in the
last full sentence before Lemma~\ref{L.2S},
by regarding $X$ as $S_0\cup S_1\cup S_2,$ straightening out one of
these circles to cover a segment of length~$2$ in a $\!1\!$-dimensional
subspace of $V,$ and letting the other two circles collapse
into that line in any way.
(Or for a construction that relies less on geometric intuition,
pick any $p\in S_0,$ map $X$ into $\R$ by the function $d(p,-),$ note
that this map sends $p$ and the point antipodal to $p$ on $S_0$ to $0$
and $2,$ respectively, and embed $\R$ in $V.)$
As before, such an image of $X$ has points $2$ units apart, and
so has radius $\geq 1$ in $V$ by the triangle inequality.
The upper bound $7/6$ was obtained in~(\ref{x.7/6}).

To see that when $V$ itself is $\!1\!$-dimensional, the value~$1$ is
not exceeded, note that the distance between any two
points of $X$ is $\leq 2.$
Hence the image of $X$ under any map into
such a $V$ is a segment of length $\leq 2,$ hence of radius $\leq 1.$

The lower bounds $1/\sqrt{3}+1/2$ and $7/6$ for $V=\R^2$
with the two indicated norms were obtained above by explicit mappings.
\end{proof}

Let us now justify the assumption we made just before the statement
of the above lemma, about the center from which we computed the radius.

\begin{lemma}\label{L.sym}
Let $V$ be a normed vector space, $A$ a nonempty subset of $V,$
and $G$ a finite group of isometries of $V$ which preserve $A.$
Then
\begin{xlist}\item\label{x.radbyVG}
$\rad_V(A)\ =\ \inf_{y\in V^G}\ \sup_{a\in A}\ d(a,y),$
\end{xlist}
where $V^G$ is the fixed-point set of~$G.$

In particular, if $V^G$ is a singleton $\{v_0\},$ then
\begin{xlist}\item\label{x.v0}
$\rad_V(A)\ =\ \sup_{a\in A}\ d(a,v_0).$
\end{xlist}
\end{lemma}\begin{proof}
Given $v\in V,$ let
\begin{xlist}\item\label{x.y=avg}
$y\ =\ |G|^{-1}\sum_{g\in G}\ g\<v,$
\end{xlist}
and note that this point lies in $V^G,$ and that for any $a\in A,$
\begin{xlist}\item\label{x.d*Sgva}
$d(a,y)\ =\ %
d(a,\<|G|^{-1}\sum_{g\in G}\<gv)\ \leq\ %
\sum_{g\in G}|G|^{-1}d(a,\<gv)\ =\ %
|G|^{-1}\sum_{g\in G}\<d(g^{-1}a,\<v)\ \\[6pt]\hspace*{.48in}\leq\ %
|G|^{-1}\sum_{g\in G}\<\sup_{b\in A}\ d(b,v)\ =\ %
\sup_{b\in A}\ d(b,v).$
\end{xlist}
Here the first inequality holds by~(\ref{x.ci}) and the second
by considering $b=g^{-1}a.$
Hence for $y\in V^G$ defined by~(\ref{x.y=avg}),
$\sup_{a\in A} d(a,y)\leq \sup_{a\in A} d(a,v),$ from
which~(\ref{x.radbyVG}) follows.
The final assertion is a special case.
\end{proof}

For $V=\R^2$ with a regular hexagon as unit circle, the group
$G$ generated by a rotation by $2\pi/3$ about any point is an isometry
of $V,$ and if we take that point to be the center of symmetry
of the set $f(X)$ we were looking at above,
$G$ preserves $f(X)$ and has that center of symmetry as unique
fixed point; so the above lemma justifies our description of the
radius of $f(X)$ in terms of distance from that point.
In the earlier computation using the Euclidean metric on $\R^2,$
we ``saw'' that the radius was measured from the center of symmetry;
this is now likewise justified by Lemma~\ref{L.sym}.

Lemma~\ref{L.3S} leaves open

\begin{question}\label{Q.7/6or}
For $V=\R^2$ under the Euclidean norm, and $X$ the $\!1\!$-skeleton
of a regular octahedron of side~$1,$ where does $\mr(X,V)$
lie within $[1/\sqrt{3}+1/2,\ 7/6]$?

For $V=\R^n,$ again with the Euclidean norm, but $n>2,$
is the answer the same?
\end{question}

Having whetted our appetite with this
example, let us prove some general results.

\section{General properties of mapping radii.}\label{S.gen}

\begin{lemma}\label{L.facts}
Let $X,$ $X',$ $Y,$ $Y'$ be nonempty metric spaces,
$\fb{Y}$ and $\fb{Y}'$ classes of such metric spaces,
and $V$ and $V'$ normed vector spaces.\\[6pt]
\textup{(i)}\ \ If there exists a surjective map $h:X\rightarrow X'$
\textup{(}or more generally, a map $X\rightarrow X'$ with
dense image\textup{)} in $\M,$ then $\mr(X',Y)\leq\mr(X,Y).$\\[6pt]
\textup{(ii)}\ \ If $Y'\subseteq Y,$ then for any nonempty
subset $A$ of $Y'$ we have $\rad_{Y'}(A)\geq\rad_Y(A).$
Here equality will hold if $Y'$ is a {\em retract} of $Y;$
i.e., if the inclusion of $Y'$ in $Y$ has a left inverse in $\M.$

Hence if $Y'$ is a retract of $Y,$ then $\mr(X,Y')\leq\mr(X,Y).$
In particular, this is true if $Y$ is a normed vector space
\textup{(}or more generally, a convex subset of such a space\textup{)}
and $Y'$ the fixed subspace \textup{(}respectively,
subset\textup{)} of a finite group $G$ of affine
isometries of $Y.$\\[6pt]
\textup{(iii)}\ \ If $\fb{Y'}\subseteq\fb{Y},$ then
$\mr(X,\fb{Y}')\leq\mr(X,\fb{Y}).$\\[6pt]
\textup{(iv)}\ \ In contrast to~\textup{(i)} and~\textup{(ii)},
for $X'\subseteq X,$ either of the numbers $\mr(X',Y)$
and $\mr(X,Y)$ can be greater than the other, and if $Y'\subseteq Y,$
or if $Y'$ is a surjective image of $Y$ in $\M,$ either of the numbers
$\mr(X,Y')$ and $\mr(X,Y)$ can be greater than the other.
\end{lemma}\begin{proof}
(i)\ \ Suppose $h:X\rightarrow X'$ has dense image.
Then for any $f:X'\rightarrow Y,$ $fh(X)$ is dense in $f(X'),$
hence $\rad_Y(fh(X))=\rad_Y(f(X')),$ so the terms of the supremum
defining $\mr(X,Y)$ include all the terms of the supremum
defining $\mr(X',Y),$ from which the asserted inequality follows.

(ii)\ \ The terms of the infimum defining $\rad_Y(A)$
include the terms of the infimum defining $\rad_{Y'}(A),$
giving the first inequality.

If there exists a retraction $e$ of $Y$ onto $Y',$ then for every
$y\in Y$ and $a\in A$ we have $d(a,e(y))\leq d(a,y),$ since $e$
is nonexpansive and fixes points of $A.$
Hence $\sup_{a\in A}d(a,e(y))\leq\sup_{a\in A}d(a,y),$ and taking
the infimum of this over $y\in Y,$ we get $\rad_{Y'}(A)\leq\rad_Y(A).$
This and the previous inequality give the asserted equality.
Since $\M(X,Y')\subseteq\M(X,Y),$
we also get $\mr(X,Y')\leq\mr(X,Y),$ as  claimed.

If $Y$ is a convex subset of a normed vector space, and $Y'$ the fixed
set of a finite group $G$ as in the final assertion, note that
the function $e(v)=|G|^{-1}\sum_{g\in G}gv$ used in the proof of
Lemma~\ref{L.sym} is nonexpansive:
\begin{xlist}\item[]
$d(e(v),e(w))~=\ ||e(v)-e(w)||~=\ ||e(v-w)||~\leq\ ||v-w||~=\ d(v,w),$
\end{xlist}
and is a retraction of $Y$ onto $Y^G.$

(iii)\ \ This is again a case of suprema of a smaller and
a larger set of real numbers.

(iv)\ \ The assertion for $X'\subseteq X$ can be seen from the following
mapping radii, where subsets of $\R$ are given the induced metric:
\begin{xlist}\item[]
$\mr(\{0\},\,\{0,2\})=0,\quad
\mr(\{0,2\},\,\{0,2\})=2,\quad
\mr(\{0,1,2\},\,\{0,2\})=0.$
\end{xlist}
The assertion for $Y'\subseteq Y$ is shown by the observations
\begin{xlist}\item[]
$\mr(\{0,2\},\,\{0\})=0,\quad
\mr(\{0,2\},\,\{0,2\})=2,\quad
\mr(\{0,2\},\,\{0,1,2\})=1.$
\end{xlist}
Finally, to get the case where $Y'$ is a surjective image of
$Y,$ note that we have surjections
$\{0,3\}\rightarrow \{0,2\}\rightarrow \{0,1\}$ in $\M,$ and that
\begin{xlist}\item[]
$\mr(\{0,2\},\,\{0,3\})=0,\quad
\mr(\{0,2\},\,\{0,2\})=2,\quad
\mr(\{0,2\},\,\{0,1\})=1.$
\end{xlist}
(With a bit more work, one can construct sets
$X,\ Y_0,\ Y_1,\ Y_2\subseteq\R$ such that each $Y_{i+1}$
is {\em both} a subset and a surjective image of $Y_i,$ and
such that $\mr(X,Y_0)<\mr(X,Y_1)>\mr(X,Y_2).)$
% $X=\{0,2\}, Y_0=\{0,1,2,5\}, Y_1=\{0,2\}, Y_2=\{0\}.$
\end{proof}

To state consequences of the above results, let us fix some notation.

\begin{definition}\label{D.diam&&}
For $n\geq 0,$ $\!n\!$-dimensional Euclidean space, i.e., $\R^n$ with
the Euclidean norm, will be denoted $\E^n.$
The class of all Euclidean spaces, $\{\E^n\mid n\geq 0\},$
will be denoted $\bf{Euc}.$

The class of all normed vector spaces, regarded as metric
spaces, will be denoted $\bf{NmV}.$
The class of all convex subsets of normed vector spaces, regarded as
metric spaces, will be denoted $\bf{Conv}.$

The {\em diameter} of a metric space $X$ will be defined by
$\diam(X)\ =\ \sup_{x,y\in X} d(x,y).$
\end{definition}

\begin{corollary}\label{C._<_<}
If $X$ is a nonempty metric space, then
\begin{xlist}\item\label{x.E1E2...}
$\mr(X,\E^1)\ \leq\ \mr(X,\E^2)\ \leq
\ \dots\ \leq\ \mr(X,\E^n)\ \leq\ \dots\ ,$
\end{xlist}
with supremum $\mr(X,\fb{Euc}).$
Further,
\begin{xlist}\item\label{x.d/2...d}
$\diam(X)/2\ =\ \mr(X,\,\E^1)\ \leq\ %
\mr(X,\,\fb{Euc})\ \leq\ \mr(X,\,\fb{NmV})\\*[6pt]\hspace*{.72in}\ %
\leq\ \mr(X,\,\fb{Conv})\ \leq\ \mr(X,\,\fb{\M})\ =\ \rad_X(X)\ %
\leq\ \diam(X).$
\end{xlist}
\end{corollary}\begin{proof}
Since $\E^n$ is the fixed subspace of a reflection
of $\E^{n+1},$ the final assertion
of Lemma~\ref{L.facts}(ii) gives~(\ref{x.E1E2...}).
(We could have put ``$\!\<0=\mr(X,\,\E^0)\leq\<\!$'' at the
left end of~(\ref{x.E1E2...}); but this would complicate some
references we will want to make to~(\ref{x.E1E2...}) later.)
By definition, $\mr(X,\fb{Euc})$ is the supremum of these values.

To see the initial equality of~(\ref{x.d/2...d}), note on the one hand
that under any nonexpansive map $f:X\rightarrow\E^1,$ the
images of any two points of
$X$ are $\leq\diam(X)$ apart, hence $f(X)$ must
lie in an interval of length $\leq\diam(X),$ and any interval in $\E^1$
has radius half its length, so $\mr(X,\,\E^1)\leq\diam(X)/2.$
On the other hand,
for $x,y\in X,$ the function $d(x,-):X\rightarrow\E^1$
is nonexpansive, and the images of $x$ and $y$ under
this map are $d(x,y)$ apart, whence the radius
of $f(X)$ is at least half this value.
Taking the supremum over all $x$ and $y,$ we get
$\mr(X,\,\E^1)\geq\diam(X)/2.$

The next four steps, inequalities among mapping radii, are
instances of Lemma~\ref{L.facts}(iii).
In the equality following these, the direction
``$\!\leq\!$'' simply says that nonexpansive
maps are radius-nonincreasing, while ``$\!\geq\!$'' holds because
one of the maps in the supremum defining
$\mr(X,\,\fb{\M})$ is the identity map of $X.$
The final inequality is immediate.
\end{proof}

We note in passing some cases where these mapping radii are
easy to evaluate.

\begin{corollary}\label{C.r=d/2}
If a metric space $X$ satisfies $\rad_X(X)=\diam(X)/2,$
then all terms of~\textup{(\ref{x.d/2...d})}
through $\rad_X(X)$ are equal \textup{(}and hence also equal to
all terms of~\textup{(\ref{x.E1E2...}))}.

In particular, this is true whenever \textup{(i)}~$X$ is a finite tree,
with edges of arbitrary positive lengths, under the arc-length metric,
or \textup{(ii)}~$X$ has an isometry $\rho$ with a fixed point $0$
such that for every $x\in X,$ $d(x,\rho(x))=2\,d(x,0).$
\end{corollary}\begin{proof}
The first sentence is clear from~(\ref{x.d/2...d}).
To get the two classes of examples, it suffices to show in
each case that $\rad_X(X)\leq\diam(X)/2,$ since~(\ref{x.d/2...d}) gives
the reverse inequality.

In case (i), $X$ is compact, so we may choose $x,y\in X$ with
$d(x,y)=\diam(X).$
The unique non-self-intersecting path between
$x$ and $y$ is isometric to a closed interval, and
so has a midpoint $p,$ satisfying $d(x,p)=d(y,p)=\diam(X)/2;$ it now
suffices to show that $d(z,p)\leq\diam(X)/2$ for all $z\in X.$
Consider the unique non-self-intersecting path from $p$ to $z.$
Because $X$ is a tree, that path out of $p$ cannot have nontrivial
intersection with both the path from $p$ to $x$ and the path
from $p$ to $y;$ assume it meets the latter only in $p.$
Then the unique non-self-intersecting path from $z$ to $y$
is the union of the path from $z$ to $p$ and the path from $p$ to $y,$
and we know that it has length $\leq\diam(X),$ so subtracting off
$\diam(X)/2,$ the length of the path from $p$ to $y,$ we conclude that
the length of the path from $z$ to $p$ is $\leq\diam(X)/2,$ as required.

In case (ii), we have $\rad_X(X)\leq\sup_{x\in X} d(x,0)=
\sup_{x\in X} d(x,\rho(x))/2\leq\diam(X)/2.$
\end{proof}

Examples falling under case~(ii) above include all centrally symmetric
subsets of normed vector spaces containing $0,$ under the
induced metric, and a hemisphere under the geodesic metric.

A less trivial result, now.
Recall that in proving the upper bounds on the mapping radii of
Lemmas~\ref{L.S}, \ref{L.2S}~and~\ref{L.3S}, we in effect
chose formal weighted combinations of points of $X,$ and used these
to specify convex linear combinations of points of $f(X)\subseteq V.$
We abstract this technique below.
In the statement of the theorem, as a convenient
way to express formal weighted combinations of points
of $X,$ we use {\em probability measures} on $X$ with finite support.
(Recall that a probability measure on $X$ is a nonnegative-valued
measure $\mu$ such that $\mu(X)=1,$ and that $\mu$ is said to
have support in a set $X_0$ if it is zero on every subset of $X-X_0.$
Apologies for the double use of ``$\!d\<\!$'' below, for the distance
function of the metric space and the ``$\!d\!$'' of integration.)

\begin{theorem}\label{T.cnv+int}
Let $X$ be a nonempty metric space.
Then
\begin{xlist}\item\label{x.mrConv}
$\mr(X,\fb{Conv})\ =\ %
\inf_\mu\ \sup_{x\in X}\ \int_{z\in X}\,d(x,z)\ d\mu(z),$
\end{xlist}
where the infimum is over
all probability measures $\mu$ on $X$ with finite support.
\end{theorem}\begin{proof}
We first prove ``$\!\leq\!$'', imitating the
argument of Lemmas~\ref{L.S},~\ref{L.2S} and~\ref{L.3S}.
We must show, for any nonexpansive
map $f:X\rightarrow C,$ where $C$ is a convex subset of a normed
vector-space $V,$ and any probability measure $\mu$ on $X$ with
finite support, that
\begin{xlist}\item\label{x.rleqint}
$\rad_C(f(X))\ \leq\ \sup_{x\in X}\ \int_{z\in X}\ d(x,z)\ d\mu(z).$
\end{xlist}

For any point $x$ of $X,$ let $\mu_x$ denote the probability
measure on $X$ with singleton support $\{x\}.$
Since the $\mu$ of~(\ref{x.rleqint}) is a probability measure
with finite support, it has the
form $c_1\<\mu_{x_1}+\dots+c_n\<\mu_{x_n},$ where
$x_1,\dots,x_n$ are points of $X,$ and $c_1,\dots,c_n$ are
nonnegative real numbers summing to $1.$
The point $y=\sum c_i\,f(x_i)$ lies in $C,$ so by definition of the
radius, the left-hand side of~(\ref{x.rleqint}) is
$\leq\sup_{x\in X}\ d\<(y,f(x))=
\sup_{x\in X}\ d\<(\sum c_i\,f(x_i),f(x)),$ which
by~(\ref{x.ci}) is $\leq\sup_{x\in X}\ \sum_i\ c_i\,d(f(x_i),f(x)),$
which, because $f$ is nonexpansive,
is $\leq\sup_{x\in X}\,\sum_i\,c_i\,d(x,x_i).$
The sum in this expression is the integral
in~(\ref{x.rleqint}), giving the desired inequality.

In proving the direction ``$\!\geq\!$'' in~(\ref{x.mrConv}),
we may assume the metric space $X$ is bounded, since otherwise
it has infinite diameter, in which case~(\ref{x.d/2...d}) tells us
that the left hand side of~(\ref{x.mrConv}) is infinite.
Assuming boundedness, we shall display a particular embedding $e$
of $X$ in a convex subset $C$ of a normed vector space $U,$ such that
$\rad_C(e(X))$ is greater than or equal to the right-hand
side of~(\ref{x.mrConv}).

Let $U$ be the space of all continuous bounded real-valued
functions on $X,$ under the $\sup$ norm, let $e:X\rightarrow U$ take
each $x\in X$ to the function $d(x,-)$ (this $e$ is easily seen to
be nonexpansive) and let $C$ be the convex hull of $e(X).$
Now for $x\in X,$ its image $e(x)=d(x,-)$
can be written $y\mapsto\int_z d(y,z)\,d\mu_x(z).$
Hence an arbitrary $u\in C,$ i.e., a convex linear combination
of these functions, will have the same form, but with $\mu_x$
replaced by a convex linear combination $\mu$ of the measures $\mu_x,$
i.e., a general probability measure $\mu$ on $X$ with finite support.
For such a function $u,$ and any $x\in X,$ the distance $d(e(x),u)$ in
$C$ is the $\sup$ norm of $u-e(x),$ which
is {\em at least} the value of $u-e(x)$ at $x\in X,$ which
is $u(x)-0=\int_z d(x,z)\,d\mu(z).$
The radius of $e(X)$ in $C$ is thus at least the infimum over all $\mu$
of the supremum over all $x$ of this integral, which is the right-hand
side of~(\ref{x.mrConv}).
\end{proof}

Recall that when we obtained our bound~(\ref{x.7/6}) on the
mapping radius of the $\!1\!$-skeleton of an octahedron, analogy and
good luck led us to the formal linear
combination of points of $X$ used in~(\ref{x.sum/6})
(in effect, a probability measure $\mu),$
which turned out to give the optimal bound.
In general we ask

\begin{question}\label{Q.mu}
Let $X$ be a finite graph with edges of possibly unequal lengths,
under the arc-length metric.
Must there be a probability measure $\mu$ on
$X$ with finite support that realizes
the infimum of~\textup{(\ref{x.mrConv})}?

Is there an algorithm for finding such a $\mu$ if it exists, or
if not, for evaluating~\textup{(\ref{x.mrConv})}?
\end{question}

We cannot expect in general that a measure of the desired sort
will have support in the set of {\em vertices} of the graph $X,$
as happened in Lemma~\ref{L.3S}.
E.g., if $X$ is isometric to a circle with arc-length metric, one can
show that a measure $\mu$ realizes the infimum of~(\ref{x.mrConv})
if and only if it gives equal weight to $p$ and $q$ whenever
$p$ and $q$ are antipodal points;
so if $X$ is, say, an equilateral polygon with an odd number of
vertices, $\mu$ cannot be concentrated in the vertices.

A class of examples generalizing our octahedral
skeleton, which it would be of interest to examine, are
the $\!1\!$-skeleta of cross polytopes~\cite{Coxeter}.

A situation simpler than that of
Question~\ref{Q.mu} is that of a {\em finite} metric space $X.$
Here the determination of the right-hand side of~(\ref{x.mrConv})
is a problem in linear programming;
whether it has an elegant solution I don't know.
The determination of $\mr(X,\E^n)$ for such a space $X$ is,
similarly, in principle, a problem in calculus.

In Corollary~\ref{C.r=d/2}, we saw that the mapping radius is easy to
compute for a space that has ``a robust center''.
Using the preceding theorem, let us
show the same for a space with a pair of ``robust antipodes''.

\begin{corollary}\label{C.antipode}
Suppose the metric space $X$ has a pair of points $p$ and $q$ such that
\begin{xlist}\item\label{x.prq}
$(\forall\,r\in X)\ \ d(p,r)+d(r,q)\ =\ d(p,q).$
\end{xlist}
Then letting $D=d(p,q),$ we have $\diam(X)=D,$ and
$\mr(X,\fb{Conv})=D/2.$
Thus, the terms of~\textup{(\ref{x.d/2...d})}
through $\mr(X,\fb{Conv})$ are all equal to $D/2.$

In particular, this is true if $X$ is the {\em $\!1\!$-skeleton}
of a regular tetrahedron or of a parallelopiped \textup{(}in
particular, of a cube\textup{)}, with the arc-length metric,
or is the {\em $\!0\!$-skeleton} of any of the regular polyhedra
other than the tetrahedron, with metric induced by the
arc-length metric on the $\!1\!$-skeleton of that polyhedron.

The property~\textup{(\ref{x.prq})}
is, of course, inherited by any subspace of $X$ containing $p$ and $q.$
\end{corollary}\begin{proof}
For any two points $r,r'\in X,$ we have
\begin{xlist}\item[]
$2\,d(r,r')\<\leq\<(d(r,p)+d(p,r'))\<+\<(d(r,q)+d(q,r'))\,=\,%
(d(p,r)+d(r,q))\<+\<(d(p,r')+d(r',q))\<=\<2D,$
\end{xlist}
so $d(r,r')\leq D,$ whence $\diam(X)=D.$
Now let $\mu$ be the probability measure giving weight $1/2$
to each of $p$ and $q.$
For this $\mu,$ the integral on the right-hand side
of~(\ref{x.mrConv}) has value $D/2$ for all $x,$ hence the supremum
of that integral over $x$ is $D/2,$ hence~(\ref{x.mrConv}) shows that
$\mr(X,\fb{Conv})\leq D/2.$
Comparing with the first term of~(\ref{x.d/2...d}), we see that
all the the terms of~(\ref{x.d/2...d}) through $\mr(X,\fb{Conv})$
(though not, as before, through $\mr(X,\fb{Metr}))$ are equal.

For $X$ the $\!1\!$-skeleton of
a regular tetrahedron, we get~(\ref{x.prq}) on taking
for $p$ and $q$ the midpoints of two opposite edges.
For $X$ the $\!1\!$-skeleton of a parallelopiped,
we can use any two antipodal points (not necessarily vertices.
In picturing this case, it may help to note that $X$ is isometric
to the $\!1\!$-skeleton of a {\em rectangular} parallelopiped.)
In the $\!0\!$-skeleton cases, we use any pair of opposite vertices.
In each case, the verification of~(\ref{x.prq}) is not hard.

The final sentence is clear.
\end{proof}

So, for instance, for the $\!1\!$-skeleta of the tetrahedron and cube
of edge $1,$ the $\!6\!$-tuples of terms of~(\ref{x.d/2...d})
(not distinguishing terms shown connected by equals-signs) are
$(1,1,1,1,\<3/2,\<2)$ and $(3/2,\,3/2,\,3/2,\,3/2,\,3,\,3)$
respectively.
(The reason the last two numbers are equal for
the cube, but distinct for the tetrahedron, is that for the cube,
the function $x\mapsto\sup_y d(x,y)$ is $3$ for all $x,$ while
for the tetrahedron, it ranges from a maximum value $2$ at the midpoints
of the edges to a minimum value $3/2$ at the vertices.
In neither of these cases is the maximum twice the minimum, so
neither of them falls under Corollary~\ref{C.r=d/2}.)

Let us note a curious feature of the construction used in
Theorem~\ref{T.cnv+int}: it has what at first looks
like a universal property (part~(i) of the next result)
but turns out not to be (part~(ii)).

\begin{corollary}[to proof of Theorem~\ref{T.cnv+int}]\label{C.notuniv}
Let $X$ be a bounded metric space, let $U$ be the space of continuous
bounded real-valued functions on $X$ under the $\sup$ norm
\textup{(}cf.\ second half of the proof of
Theorem~\ref{T.cnv+int}\textup{)}, and let $e:X\rightarrow U$ be the
map taking each $x\in X$ to the function $d(x,-).$

Now let $f:X\rightarrow V$ be any map \textup{(}in $\M)$ from $X$
into a normed vector space $V.$
Then\\[6pt]
\textup{(i)}\ \ For every family of points $x_1,\dots,x_n\in X,$
every family $c_1,\dots,c_n$ of nonnegative real numbers summing
to $1,$ and every $x\in X,$ one has
\begin{xlist}\item\label{x.g_<e}
$d(f(x),\,\sum_i c_i\,f(x_i))\ \leq\ d(e(x),\,\sum_i c_i\,e(x_i)).$
\end{xlist}

However,\\[6pt]
\textup{(ii)}\ \ Given points $x_1,\dots,x_n\in X,$ and two
families of nonnegative real numbers $b_1,\dots,b_n$
and $c_1,\dots,c_n,$ each summing to $1,$ it is {\em not necessarily}
true that
\begin{xlist}\item\label{x.gnot_<e}
$d(\sum_i b_i\,f(x_i),\,\sum_i c_i\,f(x_i))\ \leq\ %
d(\sum_i b_i\,e(x_i),\,\sum_i c_i\,e(x_i)).$
\end{xlist}
Thus, the convex hull of $e(X)$ need not admit a map
\textup{(}in $\M)$ to the convex hull of $f(X)$
making a commuting triangle with $e$ and $f.$
\end{corollary}\begin{proof}
(i) may be seen by combining the calculations of the last sentence
of the proof of the ``$\!\leq\!$'' direction of Theorem~\ref{T.cnv+int},
which shows that $d\<(f(x),\sum c_i\,f(x_i))\leq
\sum_i\,c_i\,d(x,x_i),$ and the end of the proof of the
``$\!\geq\!$'' direction, which, by evaluating $e(x_i)$ and $e(x)$
as elements of the function-space $U$ at the point $x,$ shows that
$d\<(e(x),\sum c_i\,e(x_i))\geq \sum_i\,c_i\,d(x,x_i).$

To get (ii), let $X$ again be a circle of circumference $4$ with
arc-length metric, and
let $x_0, x_1, x_2, x_3\in X$ be four points equally spaced around it.
Note that for any $u\in X,$ we have
$d(x_0,u)+d(x_2,u)=2=d(x_1,u)+d(x_3,u).$
Hence if we choose the $b_i$ and $c_i$ so that the right-hand-side
of~(\ref{x.gnot_<e}) is $d((e(x_0){+}e(x_2))/2,\linebreak[0]\,
(e(x_1){+}e(x_3))/2),$
we see that this value is $0.$
On the other hand, if we map $X$ into $\E^1$ by
$f(x)=1-\nolinebreak\max(d(x_0,x),1),$
then of the $f(x_i),$ only $f(x_0)$ is nonzero, so
the left-hand side is not $0,$ so~(\ref{x.gnot_<e}) fails.
\end{proof}

There are, in fact, a different normed vector space $U$ and mapping
$e:X\rightarrow U$ for which the universal property of~(\ref{x.gnot_<e})
does hold \cite[Theorem~2.2.4]{NW}; we examine this construction
in an appendix, \S\ref{S.AE}.

\section{Some explicit mapping radii.}\label{S.discr}
A classical result of H.\,E.\,W.\,Jung is, in effect, an
evaluation of the mapping radius in $\E^n$
of a very simple metric space.

\begin{theorem}[after Jung~\cite{J}]\label{T.Jung}
Let $D_\infty$ denote an infinite metric space in which the
distances between distinct points are all~$1.$
\textup{(}The cardinality does not matter as long as it is
infinite.\textup{)}
Then the values of $\mr(D_\infty,\,\E^n)$ for $n=0,1,2,\dots$ are,
respectively,
\begin{xlist}\item\label{x.Dinfty}
$0\ <\ 1/2\ <\ 1/\sqrt 3\ <\ \sqrt{3/8}\ <\ %
\dots\ <\ \sqrt{n/(2(n+1))}\ <\ \dots\,.$
\end{xlist}
Hence, $\mr(D_\infty,\<\fb{Euc})=1/\sqrt{2}\,.$

Likewise, for any positive integer $m,$ if we
let $D_m$ be an $\!m\!$-element metric space with all pairwise
distances~$1,$ then for every $n\geq 0,$
\begin{xlist}\item\label{x.Dm}
$\mr(D_m,\,\E^n)\ =\ \sqrt{r/(2(r+1))}\,,$\quad where $r=\min(m{-}1,n).$
\end{xlist}
Hence, $\mr(D_m,\<\fb{Euc})=\sqrt{(m-1)/(2m)}\,.$
\end{theorem}\noindent{\em Summary of proof.}
The main result of~\cite{J} is that every subset of $\E^n$ of diameter
$\leq 1$ has radius $\leq\sqrt{n/(2(n+1))}\,.$
This gives $\mr(D_\infty,\<\E^n)\leq\sqrt{n/(2(n+1))}\,.$
On the other hand, the $n+1$ vertices of the $\!n\!$-simplex of edge
$1$ in $\E^n$ form a subset of radius exactly $\sqrt{n/(2(n+1))}\,,$
and clearly
$D_\infty$ can be mapped onto that set, establishing equality.
Taking the limit of this increasing sequence as $n\rightarrow\infty,$
one gets $\mr(D_\infty,\<\fb{Euc})=1/\sqrt{2}\,.$

Clearly, the hypothesis $m>n$ works as well as $m=\infty$ in
concluding as above that $\mr(D_m,\<\E^n)=\sqrt{n/(2(n+1))}\,.$
For $m\leq n,$ on the other hand, any image of $D_m$
in $\E^n$ lies in an affine subspace that can
be identified with $\E^{m-1},$ so in that case we get
$\mr(D_m,\<\E^n)=\mr(D_m,\<\E^{m-1})\linebreak[0]=\sqrt{(m-1)/(2m)}\,.$
Combining these results, we get~(\ref{x.Dm}) and the final conclusion.
\qed\vspace{6pt}

The inequalities~(\ref{x.Dinfty}) show that
each step of~(\ref{x.E1E2...}) can be strict.
What about the steps of~(\ref{x.d/2...d})?
If we identify terms connected by equal-signs,
then~(\ref{x.d/2...d}) lists six possibly distinct values,
connected by five $\!\leq\!$-signs.
Three of these $\!\leq\!$-signs are shown
strict by the $\!3\!$-point metric space $D_3$ of
the above theorem, for which, I claim, the $\!6\!$-tuple of values
is $(1/2,\,1/\sqrt{3}\,,\,2/3,\,2/3,\,1,\,1).$
The first of these values, and the last two, are clear, and the
second comes from the above theorem (line after~(\ref{x.Dm})).
To evaluate the remaining two values, $\mr(D_3,\,\fb{NmV})$
and $\mr(D_3,\fb{Conv}),$
consider the embedding $e:D_3\rightarrow U$ as in the last paragraph
of the proof of Theorem~\ref{T.cnv+int}.
The space $U$ used there can in this case be described as $\R^3$ under
the $\sup$ norm; let $C$ be the convex hull in $U$ of
\begin{xlist}\item\label{x.011}
$e(D_3)\ =\ \{(0,1,1),\,(1,0,1),\,(1,1,0)\}.$
\end{xlist}
Then Lemma~\ref{L.facts}(ii) (in particular, the final sentence)
tells us that $\rad_C(e(D_3))$ is the common distance of the three
points of $e(D_3)$ from the unique point of $C$ invariant
under cyclic permutation of the coordinates, namely $(2/3,\<2/3,\<2/3).$
This common distance is $2/3$ (since each
member of $e(D_3)$ has a zero coordinate), so $\rad_C(e(D_3))=2/3,$
and by Theorem~\ref{T.cnv+int}, this is $\mr(D_3,\fb{Conv}).$
Since $\mr(D_3,\,\fb{NmV})\leq\mr(D_3,\,\fb{Conv}),$
to show that $\mr(D_3,\,\fb{NmV})$ is also $2/3$
it will suffice to obtain a nonexpansive map $f$ of
$D_3$ into a vector space $V$ such that $\rad_V(f(D_3))=2/3.$
This may be done by using the same mapping as above, but
translated by $(-2/3,-2/3,-2/3),$ so that the affine span of its
image becomes a vector subspace of $\R^3,$ which, with
its induced norm, we take as our $V.$
The preceding argument now gives $\rad_V(f(D_3))=2/3.$

For a space showing strict inequality at the final step
of~(\ref{x.d/2...d}), $\rad_X(X)\leq\diam(X),$
one can use any nontrivial instance of Corollary~\ref{C.r=d/2};
for instance, the unit interval $[0,1],$ for which that corollary
shows that the $\!6\!$-tuple in question
is $(1/2,\,1/2,\,1/2,\,1/2,\,1/2,\,1).$

This leaves the step
\begin{xlist}\item\label{x.NmVConv}
$\mr(X,\fb{NmV})\ \leq\ \mr(X,\fb{Conv}).$
\end{xlist}
I thought at first that equality had to hold here:
that for a $C$ a convex subset of a normed vector space $V$ and
any $A\subseteq C$ (in particular, the image of any
map of a metric space into $C),$ one had $\rad_V(A)=\rad_C(A).$
However, this is not so: consider the untranslated case~(\ref{x.011})
of the
above $D_3$ example, and note that the point $(1/2,1/2,1/2)\in U$ has
distance $1/2$ from each point of~(\ref{x.011});
so $\rad_U(e(D_3))\leq 1/2 < 2/3=\rad_C(e(D_3)).$

Nonetheless we have seen that for $X=D_3,$ equality holds
in~(\ref{x.NmVConv}).
Here, however, is an example (which it took attempts spread over many
months to find) for which that inequality is strict.

Consider the graph with 7 vertices, $x,y_0,y_1,y_2,z_0,z_1,z_2,$
and $9$ edges: a length-$\!1\!$ edge from $x$ to each of the $y_i,$ and
a length-$\!2\!$ edge from $y_i$ to $z_j$ whenever $i\neq j;$
and let $X$ be the vertex-set of this graph, with arc-length metric.
Thus, for all $i\neq j$ we have
\begin{xlist}\item\label{x.xyz}
$\!d(x,\,y_i)=1,\quad d(x,\,z_i)=3,\quad d(y_i,y_j)=2,\\[6pt]
d(y_i,z_j)=2,\quad d(y_i,z_i)=4,\quad d(z_i,z_j)=4.$
\end{xlist}

Let us first find $\mr(X,\fb{Conv}),$ using Theorem~\ref{T.cnv+int}.
We must maximize the infimum~(\ref{x.mrConv}) over the convex
linear combinations of $\mu_x,\dots,\mu_{z_2}.$
By Lemma~\ref{L.sym}, it suffices to
maximize that expression over points invariant under permutations
of the subscripts; i.e., over convex linear combinations of
\begin{xlist}\item\label{x.muxyz}
$\mu_x,\quad\mu_y=(\mu_{y_1}+\mu_{y_2}+\mu_{y_3})/3,
\quad\mu_z=(\mu_{z_1}+\mu_{z_2}+\mu_{z_3})/3.$
\end{xlist}
We find that
\begin{xlist}\item\label{x.xyzmuxyz}
$\!\mu_x(x)=0,\quad\mu_x(y_i)=\ 1,\quad\ \,\mu_x(z_i)=3,\\[6pt]
\mu_y(x)=1,\quad\mu_y(y_i)=4/3,\quad\mu_y(z_i)=8/3,\\[6pt]
\mu_z(x)=3,\quad\mu_z(y_i)=8/3,\quad\mu_z(z_i)=8/3.$
\end{xlist}
Any convex linear combination of these three
functions has value $\geq 8/3$ at each $z_i;$ so
every value of the supremum in~(\ref{x.mrConv}) is at least $8/3.$
Moreover, taking $\mu=\mu_y$ (or more generally,
$\mu=(1-t)\mu_y+t\mu_z$ for any $t\in[0,5/6]),$ we see that this
value $8/3$ is attained; so
\begin{xlist}\item\label{x.8/3}
$\mr(X,\fb{Conv})\ =\ 8/3.$
\end{xlist}

The idea of our verification that $\mr(X,\fb{NmV})$ is strictly
smaller than~(\ref{x.8/3}) will be to use
the {\em non-convex} affine combination
$(3\mu_y-\mu_x)/2$ of the functions~(\ref{x.xyzmuxyz}), so as to
reduce somewhat the highest values of $\mu_y,$ those at
the $z_i,$ without bringing the values at other points up by too much.
But since we don't have the analog of Theorem~\ref{T.cnv+int}
for non-convex combinations (and indeed, that analog is not true
in general -- if it were, then $2\mu_y-\mu_x$ would lead to a still
better result, but it does not), we must
calculate by hand rather than calling on such a theorem.
So suppose $f$ is a nonexpansive map of $X$ into a normed
vector space $V,$ and~let
\begin{xlist}\item\label{x.p=3/2-1/2}
$p\ =\ (f(y_0)+f(y_1)+f(y_2)-f(x))/2.$
\end{xlist}
We need to bound the distances between $p$ and the points of $f(X).$
In view of the symmetry of~(\ref{x.p=3/2-1/2}), it will suffice
to bound the distances to $f(x),\ f(y_0)$ and $f(z_0).$
We calculate
\begin{xlist}\item\label{x.dyf}
$\!d(p,f(x))\ =\ ||\,(f(y_0)+f(y_1)+f(y_2)-f(x)-2f(x))/2\,||\\[4pt]
\hspace*{.5in}
\leq\ (||f(y_0)-f(x)||+||f(y_1)-f(x)||+||f(y_2)-f(x)||)/2
\ \leq\ (1+1+1)/2\ =\ 3/2.\\[6pt]
d(p,f(y_0))\ =
\ ||\,(f(y_0)+f(y_1)+f(y_2)-f(x)-2f(y_0))/2\,||\\[4pt]
\hspace*{.5in}
\leq\ (||f(y_1)-f(y_0)||+||f(y_2)-f(x)||)/2\ \leq
\ (2+1)/2\ =\ 3/2.\\[6pt]
d(p,f(z_0))\ =
\ ||\,(f(y_0)+f(y_1)+f(y_2)-f(x)-2f(z_0))/2\,||\\[4pt]
\hspace*{.5in}
\leq\ (||f(y_0)-f(x)||+||f(y_1)-f(z_0)||+||f(y_2)-f(z_0)||)/2\leq
(1+2+2)/2\ =\ 5/2.$
\end{xlist}
Taking the maximum of these values, we get
\begin{xlist}\item\label{x.5/2}
$\mr(X,\fb{NmV})\ \leq\ 5/2\ <\ 8/3\ =\ \mr(X,\fb{Conv}),$
\end{xlist}
a strict inequality, as claimed.\vspace{6pt}

The above observations suggest the question:
Which normed vector spaces $V$ have the property that
the radius of every subset $X$ of $V$ is the same whether evaluated
in $V,$ or in an arbitrary convex subset of $V$ containing $X$?
This is examined in an appendix,~\S\ref{S.park}.\vspace{6pt}

The example of~(\ref{x.011}) showed that the radius
of a subset of a normed vector space could change when
one passed to a larger normed vector space.
Let us note a curious consequence.

\begin{lemma}\label{L.U0notfix}
Let $U$ be $\R^3$ under the $\sup$ norm, and
$U_0\subseteq U$ be $\{(a,b,c)\in U\mid a+b+c=0\}.$
Then there is no isometric reflection $U\rightarrow U$
having $U_0$ as its fixed subspace.
In fact, no finite group of affine isometries of any normed
vector space $W$ containing $U$ has $U_0$ as its fixed subspace.
\end{lemma}\begin{proof}
Let $W$ be any normed vector space containing $U,$ and
let $f:D_3\rightarrow U_0$ be given by $f(x)=e(x)-(2/3,2/3,2/3),$
for $e$ as in the paragraph containing~(\ref{x.011}).
The first sentence of Lemma~\ref{L.facts}(ii) gives
$\rad_W(f(D_3))\leq\rad_U(f(D_3)),$ which we saw
is $<\rad_{U_0}(f(D_3)).$
On the other hand, if $W$ had a finite group $G$ of affine isometries
with fixed subspace $U_0,$ then Lemma~\ref{L.sym} would give
$\rad_W(f(D_3))=\rad_{U_0}(f(D_3)).$
\end{proof}

Returning to~(\ref{x.E1E2...}) and~(\ref{x.d/2...d}), let us for
simplicity reduce the number of independent values by ``normalizing''
to the case $\diam(X)=2,$ and ask for more detailed information
than those inequalities.

\begin{question}\label{Q.whatsset}
Let $X$ run over all metric spaces of diameter~$2.$
What can one say about the geometry of the resulting sets of sequences
\begin{xlist}\item\label{x.infseq}
$\{(\mr(X,\E^1),\ \mr(X,\E^2),\ \dots,\ \mr(X,\E^n),\ \dots)\}
\ \subseteq\ \R^{\mathbb N},$
\end{xlist}
\begin{xlist}\item\label{x.4-term}
$\{(\mr(X,\,\fb{Euc}),\ \mr(X,\,\fb{NmV}),\ %
\mr(X,\,\fb{Conv}),\ \mr(X,\,\fb{\M}))\}
\ \subseteq\ \R^4$?
\end{xlist}

Can one describe them exactly?
Are they convex; or do they become convex
on replacing the entries by their logarithms,
or under some other natural change of coordinates?

If two successive terms of a member of~\textup{(\ref{x.infseq})} are
equal, is the sequence constant from that point on?
\end{question}

Another family of questions, suggested by Theorem~\ref{T.Jung}, is

\begin{question}\label{Q.whatvals}
For $n\geq 2,$ what can one say about the set of nonnegative real
numbers that can be written $\mr(X,\E^n)$ for finite metric
spaces $X$ in which all distances are {\em integers}?

Are all such real numbers ``constructible'', i.e., obtainable from
rational numbers by a finite sequence of square roots and
ring operations?

Is this set well-ordered for each $n$?
\textup{(}It has a smallest element $0,$ and a
next-to-smallest element $1/2.)$

Does this set change if ``finite metric spaces $X$\dots''
is weakened to ``bounded metric spaces $X$\dots''?

For $m<n,$ can one assert any inclusion between the sets of mapping
radii into $\E^m,$ and into $\E^n$?
Are there values that occur as $\mr(X,\fb{Euc})$ for some $X,$
but not as $\mr(X',\E^n)$ for any $X'$ and $n$?
\textup{(}E.g.,
can $1/\sqrt{2}$ be written in the latter form?\textup{)}
\end{question}

We end this section with an observation made in~\cite{H+D}
for the spaces $\E^n,$ which in fact holds for closed convex
subsets of arbitrary finite-dimensional normed spaces.

\begin{lemma}[cf.\ {\cite[Proposition~29, p.14, and second paragraph of
p.46]{H+D}}]\label{L.helly}
If $C$ is a closed convex subset of a normed vector space $V$ of
finite dimension $n,$ and $A$ a subset of $C$ with $>n$ elements, then
$\rad_C(A)=\sup_{A_0}\rad_C(A_0),$
where $A_0$ runs over the $\!n{+}1\!$-element subsets of $A.$
\end{lemma}\begin{proof}
``$\!\geq\!$'' is clear; so it suffices to show that if for some
real number $r,$ each $A_0$ is contained in a closed ball of radius
$r$ centered at a point of $C,$ then so is $A.$
Now for each $a\in A,$ the set of $v\in C$ such that $a$ lies
in the closed ball in $C$ of radius $r$ about $v$ is the closed ball
in $C$ of radius $r$ about $a,$ hence a compact convex subset of $V.$
To say that a set $A_0$ is contained in some closed ball of radius $r$
centered at a point of $C$ is to say that the intersection of these
sets, as $a$ runs over $A_0,$ is nonempty.
By Helly's Theorem (\cite{H}, \cite{DGK}),
if a family of compact convex subsets
of $\R^n$ has the property that every system of $n+1$ members of this
family has nonempty intersection, then so does the whole family;
which in this case means that all of $A$ is contained in a ball
of the indicated sort.
\end{proof}

The above lemma does not imply the corresponding statement
for mapping radii.
For example, let $X=\{x,y_0,y_1,y_2\},$ where $x$ has distance $1/2$
from each of the $y_i,$ and these have distance $1$ from each other.
The maximum of the mapping radii in $\E^2$ of $\!3\!$-element subsets
of $X$ is $\mr(\{y_0,y_1,y_2\},\E^2)=\mr(D_3,\E^2)=1/\sqrt{3}\<.$
But $\mr(X,\E^2)\leq\rad_X(X)=1/2.$

On the other hand, for this example, $\mr(X,\E^2)$
can be described as the infimum over $p\in X$ of the supremum
of $\mr(X_0,\E^2)$ over all $\!3\!$-element subsets
$X_0$ of $X$ containing $p.$
So we ask

\begin{question}\label{Q.helly}
Does there exist, for every positive integer $n,$ a positive
integer $N$ and a formula which for every metric space $X$
of $\geq N$ elements, and every normed vector space $V$ of dimension
$n,$ expresses $\mr(X,V),$ using the operations of
suprema and infima, in terms of the numbers $\mr(X_0,V),$
for $\!N\!$-element subsets $X_0\subseteq X$?
\end{question}

\section{Realizability of mapping radii.}\label{S.realize}

For a subset $A$ of a metric space $Y,$ let us say that
$\rad_Y(A)$ is {\em realized} if
the infimum in the definition~\textup{(\ref{x.rad})} of that
expression is attained, that is, if there exists $y\in Y$ such that
$A$ is contained in the closed ball of radius $\rad_Y(A)$ about $y.$

Likewise, for metric spaces $X$ and $Y,$ let us say that $\mr(X,Y)$
is realized if the supremum in the definition of that expression is
attained; that is, if there exists an $f:X\rightarrow Y$ such
that $\rad_Y(f(X))=\mr(X,Y).$
(This does not presume that $\rad_Y(f(X))$ is realized.)

\begin{lemma}\label{L.compact}
Let $X$ and $Y$ be nonempty metric spaces.\\[6pt]
\textup{(i)}\ \ If $Y$ is compact,
then for any subset $A\subseteq Y,$ $\rad_Y(A)$ is realized.\\[6pt]
\textup{(ii)}\ \ If $X$ and $Y$ are both compact,
then $\mr(X,Y)$ is realized.\vspace{6pt}

However\\[6pt]
\textup{(iii)}\ \ For $X$ compact and $Y$ bounded and complete,
or for $X$ bounded and complete
and $Y$ compact, $\mr(X,Y)$ may fail to be realized.
\end{lemma}\begin{proof}
(i) follows from the fact that for bounded $A,$
$\sup_{a\in A}\ d(a,y)$ is a
continuous function of $y,$ hence assumes a minimum on $Y.$

To get~(ii), we note that $\M(X,Y)$ is a closed subset of the function
space $Y^X,$ which is compact because $Y$ is,
so $\M(X,Y)$ is compact in the function topology.
We would like to say that the real-valued map on this space given
by $f\mapsto\rad_Y(f(X))$ is continuous, and hence assumes a maximum.
For general $X,$ this continuity does not hold, as will follow from the
second statement of~(iii); but I claim that it holds if $X$ is compact.
For given $f\in\M(X,Y)$ and $\varepsilon>0,$ compactness allows us to
cover $X$ by finitely many open balls of radius $\varepsilon/3,$
say centered at $x_1,\dots,x_n.$
Consider the neighborhood of $f$ in $\M(X,Y)$ given by
\begin{xlist}\item[]
$U\ =\ \{g\in\M(X,Y)\mid\ d(f(x_i),g(x_i))<\varepsilon/3\ \ %
(i=1,\dots,n)\}.$
\end{xlist}
Taking any $x\in X$ and $y\in Y,$ note that there exists $i$ such
that $d(x_i,x)<\varepsilon/3;$ hence for $g\in U,$
\begin{xlist}\item[]
$|d(f(x),y)-d(g(x),y)|\ \leq\ d(f(x),g(x))\ \leq\ %
d(f(x),f(x_i))+d(f(x_i),g(x_i))+d(g(x_i),g(x))\\*[6pt]\hspace*{1.47in}
\leq\ \varepsilon/3+\varepsilon/3+\varepsilon/3\ =\ \varepsilon.$
\end{xlist}

Thus, the two functions associating to every
$y\in Y$ the numbers $\sup_{x\in X} d(f(x),y)$ and
$\sup_{x\in X} d(g(x),y)$ differ everywhere by $\leq\varepsilon,$
whence the infima of these functions, $\rad_Y(f(X))$ and $\rad_Y(g(X))$
differ by $\leq\varepsilon,$ giving continuity of
$f\mapsto\rad_Y(f(X)),$ which, as noted above, yields~(ii).

(iii)\ \ For an example with $X$ but not $Y$ compact,
let $X=D_2,$ i.e., a space consisting of two points at
distance $1$ apart, and let $Y=\{y_2,y_3,\dots,y_n,\linebreak[0]
\dots\}\cup\{z\},$ with $d(y_m,y_n)=1$ $(m\neq n)$
and $d(y_n,z)=1-1/n.$
Note that the radius in $Y$ of a point-pair $\{z,y_n\}$
or $\{y_m,y_n\}$ with $m<n$ is $1-1/n.$
Now $\M(X,Y)$ consists of {\em all} set-maps $X\rightarrow Y,$ and it
follows from the above calculation that
$\mr(X,Y)=\linebreak[3]\sup_n (1-1/n)=1,$
but that this value is not achieved.
(If we had not specified that $Y$ should be complete, we
could have used the simpler example, $X=D_2,\ Y=[0,1).)$

For an example with $Y$ but not $X$ compact,
let $X=\{x_2,x_3,\dots,x_n,\dots\},$ with
$d(x_{2n},x_{2n+1})=2-1/n,$ and all other pairs of distinct
points having distance~$1;$ and let $Y=[0,2]\subseteq\E^1.$
Note that if a map $X\rightarrow Y$ is to have radius $>1/2,$
it must send some pair of points to values differing by $>1,$
and by our metric on $X,$ these two points must have the
forms $x_{2n},\ x_{2n+1}.$
Since all other points have distance $1$ from these two, the images of
all other points must fall within the interval between their images.
Hence the image of our map falls within an interval of
length $\leq 2-1/n$ for some positive $n,$ i.e., of length
$<2,$ and hence of radius $<1.$
But such images can have radii arbitrarily close to $1,$
again giving a mapping radius that is not realized.
\end{proof}

\begin{corollary}\label{C.loc_cmpct}
Suppose $Y$ is a metric space in which every
closed bounded subset is compact.
Then\\[6pt]
\textup{(i)}\ \ For every bounded nonempty
subset $A\subseteq Y,$ $\rad_Y(A)$ is realized.\\[6pt]
\textup{(ii)}\ \ If the isometry group of $Y$ is transitive,
or more generally, if $Y$ has a bounded subset which meets
every orbit of that group, then
for every compact nonempty metric space $X,$ $\mr(X,Y)$ is realized.
\end{corollary}\begin{proof}
(i)\ \ Let $\rad_Y(A)=r,$ choose any $a_0\in A,$ and let $Y'$ be
the closed ball of any radius $r'>\diam(A)\geq r$ about $a_0$ in $Y.$
By assumption $Y'$ is compact.
We see that $A\subseteq Y',$ and that every point $y\in Y$
with $\sup_{a\in A} d(a,y)\leq r'$ lies in $Y'.$
Since $\inf_{y\in Y}\sup_{a\in A} d(a,y)=r,$ the space $Y$
contains points $y$ for which $\sup_{a\in A} d(a,y)$ comes
arbitrarily close to $r;$
hence it will contain points for which that value is arbitrarily
close to $r$ {\em and} is $\leq r'.$
Points with this latter property lie in $Y',$
whence $\rad_{Y'}(A)$ is also equal to $r,$ and
applying part~(i) of the preceding lemma with $Y'$ for $Y$
gives the desired conclusion.

(ii)\ \ Suppose every orbit of the isometry group of $Y$ meets the
closed ball of radius $c$ about $y_0\in Y,$
and let $x_0$ be any point of $X.$
Then every $f: X\rightarrow Y$ may be adjusted by an isometry of $Y$
(which will preserve the radius of $f(X))$ so that we get
$d(f(x_0),y_0)\leq c,$ and after this adjustment, $f(X)$ will lie
in the closed ball of radius $c+\diam(X)$ about $y_0.$
Letting $Y'$ denote the closed ball of any radius $r'>c+\diam(X)$
about $y_0,$ we see as in the proof of~(i) that the radii of these
image sets $f(X)$ in $Y'$ will equal their radii within $Y,$
and applying part~(ii) of the preceding lemma with $Y'$ for $Y,$ we
get the desired conclusion.
\end{proof}

\section{Related literature (and one more question).}\label{S.cite}

Lemma~\ref{L.S} above, determining the mapping radius of a circle in
a normed vector space $V,$ occurs frequently in the literature
(with $\E^3$ or $\E^n$ for $V)$ as an offshoot of the proof
of Fenchel's Theorem, the statement that the total curvature of a
closed curve $C$ in $\E^3$ is at least~$2\pi,$ with equality only
when $C$ is planar and convex~\cite[Satz~I\<]{F}.
To prove that theorem, Fenchel noted that this total curvature is the
length of the curve in the unit sphere $S^2$ traced by the unit
tangent vector to $C,$ and that that curve cannot lie wholly in an open
hemisphere of $S^2$ (nor in a closed hemisphere unless $C$ is planar).
He completed the proof by showing~\cite[Satz~I$\!'$]{F} that a closed
curve of length $<2\pi$ (respectively, equal to $2\pi)$ in $S^2$ must
lie in an open hemisphere (respectively, must either lie in
an open hemisphere or be a union of two great semicircles).
In our language, this says that a circle of circumference $<2\pi,$ made
a metric space using arc-length, has mapping radius $<\pi/2$ in $S^2$
and (along with some additional information) that
the circle of arc-length exactly $2\pi$ has mapping radius $\pi/2.$

Subsequent authors~\cite{C+K}, \cite[Lemma on p.30]{Chern},
\cite{RAH}, \cite{R+HS}, \cite{BS}, \cite{PCT} gave simpler proofs
of Fenchel's Satz~I$\!'\!$ (similar to our proof of
Lemma~\ref{L.S}), and/or generalized that result from $S^2$ to $S^n,$
and/or obtained the more precise result that
the mapping radius of a circle of length $L\leq 2\pi$ in $S^n$ is
$L/4,$ and/or noted that the same method also gives the analogous result
with $\E^n,$ or indeed any of a large class of geometric structures,
in place of $S^n.$

The last-mentioned generalizations were based on the observation that
the concept of the {\em midpoint} of a pair of
points can be defined, and behaves nicely, in many geometric contexts.
I do not know whether more general convex linear
combinations, such as we used in~(\ref{x.ci}) and in the proof of
Theorem~\ref{T.cnv+int}, can be defined outside the context
of vector spaces so as to behave nicely; hence the emphasis in this
note on vector spaces and their convex subsets.
A.\<Weinstein (personal communication) suggests that an
approach to ``averaging'' of points introduced by Cartan and
developed further by Weinstein in \cite{AW} might serve this function.
J.\<Lott (personal communication) points similarly to the concepts
of Hadamard space~\cite{BH} and Busemann convex space~\cite{HB}.

The results on closed curves of length $L$ in the unit
sphere cited above all take $L\leq 2\pi.$
If we write $S^1_L$ for a circle of circumference $L$ with
arc-length metric, and $S^n$ for
the unit $\!n\!$-sphere (of circumference $2\pi)$ with geodesic
distance as metric, it is clear that the result $\mr(S^1_L,S^n)=L/4$
cannot be expected to hold when $L>2\pi;$ but it would be
interesting to investigate how that mapping radius does behave
as a function of $L.$
For all $L,$ $\mr(S^1_L,S^n)<\pi,$ since a curve of fixed length cannot
come arbitrarily close to every point of $S_n,$ and if it misses
the open disk of geodesic radius $r$ about a point $p,$ then it
is contained in the closed disc of geodesic
radius $\pi-r$ about the antipodal point.

Many of the papers referred to above consider arcs as well as closed
curves; i.e., also study $\mr([0,L],\linebreak[0]S^n),$
and prove that for $L\leq\pi,$ this equals $L/2.$
Again, the case of larger $L$ would be of interest.
So we ask

\begin{question}\label{Q.sphere}
For fixed $n>1,$ how does $\mr(S^1_L,S^n)$ behave
for $L>2\pi,$ and how does $\mr([0,L],\linebreak[0]S^n)$
behave for $L>\pi,$ as functions of $L$?

For instance, are these two functions piecewise analytic?
\end{question}

It seems likely that there will be ranges of values of $L$ in which
different configurations of a closed curve or arc give
maximum radius, and that the value of this radius will be an analytic
function of $L$ within each such range.
(I conjecture that for all $L$ between $2\pi$ and a value somewhat
greater than $3\pi,$ $\mr(S^1_L,S^2)$ will be realized by a
``$\!3\!$-peaked crown'', consisting of $6$ arcs of great circles,
with midpoints equally spaced along a common equator.
% This should become less than optimal when the largest disc that fits
% between successive edges is larger than the disc determined by the
% $3$ peaks around each pole.  Before this happens, successive edges
% must be able to clamp a disc between them, which becomes true when
% maximum distance between them is achieved at interior, not end.  It is
% achieved at distance $\pi/2$ along each, so each must have length
% $>\pi/2,$ so $L>3\pi.$  When disc clamped between them becomes
% bigger than disc around poles, next configuration probably involves
% clipping peaks of crown, to narrow disc between.  Won't work for long.
For $\mr(S^1_L,S^n)$ with $n>2,$ I have no guesses.)

Many papers in this area also consider the smallest
``box'' -- in various senses -- into which one can fit all curves,
or closed curves, of unit length \cite{C+K} \cite{snake} \cite{S+W},
or all point-sets of unit diameter \cite{DG}.
These do not translate into statements about our concept of
mapping radius for two reasons.
First, they deal with arc length in the Euclidean metric, but with
``boxes'' which, though they could in many cases be considered closed
balls in another metric, are not balls in the Euclidean
metric; and our formalism of mapping radius
does not look at more than one metric on $Y$ at a time.
Second, they generally allow rotations as well as translations in
fitting the box around the curve, while in looking at radii we
only have one closed ball of each radius centered at a given point.
\vspace{6pt}

The intuitive interest of Question~\ref{Q.sphere} above
arises in part from a special property of the sphere: that a large
open or closed ball, i.e., one that falls just short of covering $S^n,$
has for complement a small closed or open ball.
For spaces $Y$ not having this property, the most natural analogs
of those questions might be the corresponding questions
about ``mapping co-radii'', given by the definitions
\begin{xlist}\item\label{x.corad}
$\mathrm{corad}_Y(A)\ =\ \sup_{y\in Y}\ \inf_{a\in A}\ d(a,y)$
\quad$(A\subseteq Y),$
\end{xlist}
\begin{xlist}\item\label{x.mcorad}
$\mathrm{map\mbox{-}corad}(X,Y)\ =\ %
\inf_{f\in\M(X,Y)}\ \mathrm{corad}_Y(f(X))\ \\[6pt]
\hspace*{1.12in}
=\ \inf_{f\in\M(X,Y)}\ \sup_{y\in Y}\ \inf_{x\in X}\ d(f(x),y)$
\end{xlist}
(cf.~(\ref{x.rad}) and~(\ref{x.ir})).
So, for instance, one might ask about the values of
$\mathrm{map\mbox{-}corad}([0,L],B^2)$ for $B^2$ the closed unit
disc in $\R^2,$ as a function of $L.$

(I'm not sure that ``co-radius'' is a good choice of term:
one could argue that that term would more appropriately
apply either to $\rad_Y(Y{-}A),$ or to what in the notation
of~(\ref{x.corad}) would be written
$\mathrm{corad}_Y(Y\,{-}\,A).$
So the above names are just suggestions,
which others may choose to revise.)

\section{Appendix: The Arens-Eells space of $X.$}\label{S.AE}

At the end of \S\ref{S.gen}, I mentioned that every metric space $X$
admits an embedding in a normed vector space $U$ having
the universal property that Corollary~\ref{C.notuniv}(ii) showed that
the embedding we were considering there did {\em not} have.
The construction in
question was introduced by Arens and Eells~\cite{AE}, and
its universal property noted by Weaver~\cite[Theorem~2.2.4]{NW},
who calls it the Arens-Eells space of $X.$
Weaver is there most interested in this space as
a pre-dual to the Banach space of Lipschitz functions on $X.$
I will sketch below a motivation for the same object in terms
of the universal property.
My description will also make a couple of
technical choices different from those of~\cite{AE} and~\cite{NW}.

Essentially the same construction arises in mathematical economics,
in the study
of the ``transportation problem'' \cite{DG2}, cf.\ \cite[\S2.3]{NW}.
What to us will be the norm of an element of the Arens-Eells space
appears there as the minimum cost of transporting goods from a given
set of sources to a given set of markets.
\vspace{6pt}

To lead up to the construction, let a metric space $X$ be given,
consider any map (as always, nonexpansive) $f$ of $X$ into a
normed vector space $V,$ and let us ask, as a sample question: If we
know the distances among four points $x_1,\ x_2,\ x_3,\ x_4\in X,$
what can we say about $||f(x_1)+f(x_2)-f(x_3)-f(x_4)||$?

Clearly, this will be bounded above by
$||f(x_1){-}f(x_3)||\,+\,||f(x_2){-}f(x_4)||
\leq d(x_1,x_3)+d(x_2,x_4).$
The other way of pairing terms of opposite
sign similarly gives the bound $d(x_1,x_4)+d(x_2,x_3).$
Hence

\begin{xlist}\item\label{x.x1234}
$||\<f(x_1)+f(x_2)-f(x_3)-f(x_4)\<||~ \leq~
\min\<(d(x_1,x_3)+d(x_2,x_4),\ d(x_1,x_4)+d(x_2,x_3)).$
\end{xlist}

For a similar, but slightly less straightforward case, suppose we want
to bound $||\<3f(x_1)+f(x_2)-2f(x_3)-2f(x_4)||.$
We cannot, as before, pair off terms whose coefficients
in this expression happen to be the same except for sign.
There are, however, ways of breaking up that expression as a
linear combination of differences; and a little experimentation
shows that all ways
of doing so are convex combinations of two extreme decompositions.
These two cases lead to the bound
\begin{xlist}\item\label{x.x1234'}
$||\<3f(x_1)+f(x_2)-2f(x_3)-2f(x_4)\<||~ \leq\\[4pt]{}\hspace{1in}
\min\<(2\<d(x_1,x_3)\<{+}\<d(x_1,x_4)\<{+}\<d(x_2,x_4),\ \,
2\<d(x_1,x_4)\<{+}\<d(x_1,x_3)\<{+}\<d(x_2,x_3)).$
\end{xlist}

We will not stop here to prove that~(\ref{x.x1234})
and~(\ref{x.x1234'}) are best bounds.
Let us simply observe that these considerations suggest that
the norm of such a linear combination of images of points
of $X$ under a {\em universal} map
$e:X\rightarrow U$ should be given by an infimum of linear
combinations of the numbers $d(x,y)$ $(x,y\in X)$ with nonnegative
real coefficients, the infimum being taken
over all such linear expressions which, when
each $d(x,y)$ is replaced by $e(x)-e(y),$ give the required element.

An obvious problem is that the only elements we get in this
way are those in which the sum of the coefficients of the
members of $e(X)$ is $0.$
This difficulty is intrinsic in the situation:  There will not in fact
exist a nonexpansive map of $X$ into a normed vector space
having the standard sort of universal mapping
property with respect to such maps,
because, though the condition of nonexpansivity bounds
the distances among images of points of $X,$ it does not
bound the distances between such images and $0;$
so universality would force the images of points of $X$
to have infinite norm.

What we can get, rather, is a set-map $e$ of $X$ into a vector space
$U,$ and a norm {\em on the subspace} $U_0$ of linear combinations of
images of points of $X$ with coefficients summing to $0,$ such that for
all $x,y\in X,$ $||e(x)-e(y)||\leq d(x,y),$ and which has
the universal property that given any nonexpansive map $f$ of $X$ into a
normed vector space $V,$ there exists a unique vector-space homomorphism
$g:U\rightarrow V$ which satisfies $f=ge,$ and is nonexpansive
{\em on $U_0.$}
Observe that the norm on $U_0$ induces a metric on each coset of
that subspace; in particular, on the coset $U_1$ of elements in which
the sum of all coefficients is $1,$ which is the affine span
of the image of $X.$
The map of $X$ into that coset is nonexpansive, and the asserted
universal property of $e$ is easily seen to
yield~(\ref{x.gnot_<e}), the property that the construction
of \S\ref{S.gen} failed to have.

Weaver's answer to the same distance-to-$\!0\!$ problem is to use
metric spaces with basepoint, and basepoint-respecting
maps, the basepoint of a vector space being $0.$
This has the advantage of giving a universal property in the
conventional sense, with both $U$ and $V$ in the category of
normed vector spaces.
However, it requires one to make a possibly unnatural
choice of basepoint in $X;$ changes in that choice induce isometries
on the universal space, which, though affine, are not linear.
The approach I actually find most natural is to regard
what I have called $U_1$ as a ``normed affine space'', that is,
a set with a simply transitive group of ``translation'' maps by
elements of a normed vector space, and to note that $U_1$ has a genuine
universal property in the category of normed affine spaces.
However, the development
of that concept would be an excessive excursion for this appendix.
Still another approach would be to work with ``normed'' vector spaces
where the norm is allowed to take on the value $+\infty.$
In any case, it is straightforward to verify that the
Arens-Eells space of $X$ as described in \cite{NW} and my $U_0$ are
isometrically isomorphic, so below I will quote results of Weaver's,
tacitly restated for my version of the construction.

The details, now: let $U$ be the vector space of all
real-valued (i.e., {\em not} necessarily nonnegative)
measures $\mu$ on $X$ with finite support, and, as before, for each
$x\in X$ let $\mu_x$ be the probability measure with support $\{x\}.$
Thus, $\{\mu_x\mid x\in X\}$ is a basis of $U.$
Let $U_0\subseteq U$ denote the subspace of measures $\mu$
satisfying $\mu(X)=0.$
Let $W$ similarly denote the space of all real-valued
measures on $X\times X$ with finite support;
for each $(x,y)\in X\times X,$ let $\nu_{x,y}$ be the probability
measure with support $\{(x,y)\},$ and let $W_+\subseteq W$ be the cone
of nonnegative linear combinations of the $\nu_{x,y},$
i.e., the nonnegative-valued measures on $X\times X.$
Finally, let $D: W\rightarrow U$ be the linear map defined by
the condition
\begin{xlist}\item\label{x.defD}
$D(\nu_{x,y})\ =\ \mu_x-\mu_y$ \quad for $x,y\in X,$
\end{xlist}
which clearly has image $U_0.$
We now define the norm of any $\mu\in U_0$ by
\begin{xlist}\item\label{x.univnorm}
$||\mu||~ =~ \inf_{\strut\nu\in W_+,\,D(\nu)=\mu}
\ \int_{\strut(x,y)\in X\times X} d(x,y)\ d\nu\,.$
\end{xlist}
It is easy to verify that this indeed gives a norm with
the desired universal property.
The one verification that is not immediately obvious is that it
is a norm rather than a pseudonorm; i.e., that it is nonzero for
nonzero $\mu\in U_0.$
To get this, one first proves the desired universal property in the
wider context of pseudonormed
vector spaces, then notes that given any nonzero
$\mu=\sum_I a_i\,\mu_{x_i}\in U_0$ $(I$ finite, all $a_i$ nonzero), one
can find a nonexpansive map $f:X\rightarrow\R$ which is zero at all
but one of the $x_i,$ say $x_{i_0},$
from which it follows by the universal property
that $||\mu||\geq |\sum\,a_i\ f(x_i)|=|a_{i_0}\ f(x_{i_0})|>0.$

Weaver \cite[Theorem~2.3.7(b)]{NW} shows that the infimum
in~(\ref{x.univnorm}) is always attained, and in fact, by a $\nu$
whose ``support'' in $X$ (the set of points which appear as $x$ or
$y$ in terms $\nu_{x,y}$ having nonzero coefficient
in the expression for $\nu)$ coincides
with the support of $\mu$ (the set of $x$ such that $\mu_x$
appears with nonzero coefficient in the expression for $\mu).$
Our next proposition strengthens this result a bit.
For brevity, we will call on Weaver's result in the proof, but
I will sketch afterward how the argument can be made self-contained.

We will use the following notation and terminology.
Given $\nu=\sum_{j\in J} b_j\,\nu_{x_j,y_j}\in W$ (where $J$ is a
finite set, the pairs $(x_j,y_j)$ for $j\in J$ are distinct,
and all $b_j\neq 0),$ let $\Gamma(\nu)$
be the directed graph having for vertices all points of $X,$
and for directed edges the finitely many pairs $(x_j,y_j)$ $(j\in J).$
Let us define the {\em positive support} of a directed graph $\Gamma$
as the set of vertices which are initial points of its
edges, and its {\em negative support} as
the set of vertices which are terminal points.
For $\nu\in W,$ we will call the positive and negative supports
of $\Gamma(\nu)$ the positive and negative supports of $\nu.$
On the other hand, for $\mu=\sum_I a_i\,\mu_{x_i}\in U_0,$ let us
define its positive support to be $\{x_i\mid a_i>0\},$
and its negative support to be $\{x_i\mid a_i<0\}.$
These are clearly disjoint.
Note that when $\nu\in W_+,$ the positive support of $D(\nu)$
is contained in the positive support of $\nu,$ and contains
all elements thereof that are not also in the
negative support of $\nu,$ and that the negative support of $D(\nu)$
has the dual properties.

When we speak of a {\em cycle} in a directed graph, we shall mean
a cycle in the corresponding undirected graph; we shall also
understand that in a cycle no vertex is traversed more than once.
Note that a cycle of length $1$ in $\Gamma(\nu)$ can only arise when
a term $\nu_{x,x}$ has nonzero coefficient in $\nu,$
while a cycle of length
$2,$ i.e., the presence of two edges between $x$ and $y,$ can only occur
if $\nu_{x,y}$ and $\nu_{y,x}$ both have nonzero coefficients.
But a cycle of length $n>2$ involving a given sequence of vertices may
arise in any of $2^n$ ways, depending on the orientations of the edges.

We now prove

\begin{proposition}[cf.~{\cite[Theorem 3.3, p.84]{DG2}}]\label{P.graph}
Let $\mu\in U_0.$
Then the infimum of~\textup{(\ref{x.univnorm})}
is attained by an element $\nu\in W_+$ \textup{(}not
necessarily unique\textup{)} whose positive and negative
supports coincide respectively with the positive and negative supports
of $\mu,$ and whose graph $\Gamma(\nu)$ has no cycles.
\end{proposition}\begin{proof}
As mentioned, Weaver proves the existence of a $\nu\in W_+$
with $D(\nu)=\mu$ which achieves the infimum~(\ref{x.univnorm})
and has the same support as $\mu.$
Let $\nu$ be chosen, first, to have these properties;
second, among such elements, to minimize the total number of
edges in $\Gamma(\nu),$ and finally, to minimize the sum of the
coefficients of all the $\nu_{x,y}$ in its expression.
This last condition is achievable because the set of elements
of $W_+$ which are linear combinations of a given finite family
of the $\nu_{x,y},$ and for which the coefficients of these elements
are all $\leq$ some constant, is compact; so after finding some
$\nu$ with $D(\nu)=\mu$ which achieves the minimum~(\ref{x.univnorm}),
has the same support as $\mu,$ and minimizes the number of
edges in $\Gamma(\nu),$ we may restrict our search for elements also
minimizing the coefficient-sum to the compact set of elements having
these properties and having every coefficient less than
or equal to the coefficient-sum of the element we have found.

Suppose, now, that $\Gamma(\nu)$ has a cycle.
Thus, we may choose distinct vertices $p_1,\dots,p_k,$ and for each
$j\in\{1,\dots,k\},$ a term $\nu_{p_j,p_{j+1}}$ or
$\nu_{p_{j+1},p_j}$ occurring with positive coefficient in $\nu,$
where the subscripts $j$ are taken modulo~$k.$
(If $k>2$ and both $\nu_{p_j,p_{j+1}}$ and $\nu_{p_{j+1},p_j}$ occur
in $\nu,$ we choose one of these arbitrarily.
If $k=2,$ we make sure that the terms we choose for
$j=1,\,2$ are distinct, one
being $\nu_{p_1,p_2}$ and the other $\nu_{p_2,p_1}.)$
For each $j\in\{1,\dots,k\}$ let us now define $\nu'_{p_j,p_{j+1}}$
to be $\nu_{p_j,p_{j+1}}$ if that is the $\!j\!$th term in the list
we have chosen, or $-\nu_{p_{j+1},p_j}$ if the $\!j\!$th term
in that list is $\nu_{p_{j+1},p_j},$
and let $\nu'=\sum_j\nu'_{p_j,p_{j+1}}\in W.$
In general, $\nu'\notin W_+,$ but for all $\lambda\in\R$ near enough
to $0,$ we have $\nu+\lambda\<\nu'\in W_+,$
since the relevant coefficients in $\nu$ are strictly positive.
Note that for each $j,$ $D(\nu'_{p_j,p_{j+1}})=\mu_j-\mu_{j+1},$ hence
$D(\nu')=0,$ hence $D(\nu+\lambda\nu')=D(\nu)=\mu.$

Clearly, $\int_{(x,y)\in X\times X} d(x,y)\,d(\nu{+}\lambda\nu')$
is an affine function of $\lambda.$
Hence it must be constant, otherwise, using small $\lambda$ of
appropriate sign, we would get a contradiction to the assumption that
$\nu$ achieves the minimum of~(\ref{x.univnorm}); so all
the elements $\nu+\lambda\<\nu'$ achieve this same minimum.
Now some choice of $\lambda$ will
cause $\lambda\<\nu'$ to exactly cancel the smallest among the
coefficients of terms $\nu_{p_j,p_{j+1}}$ or
$\nu_{p_{j+1},p_j}$ in our cycle in $\Gamma(\nu).$
Thus, $\nu+\lambda\<\nu'$ contradicts the minimality assumption
on the number of edges in $\Gamma(\nu).$
This contradiction shows that $\Gamma(\nu)$ has no cycles.

Next, let us compare the positive and negative supports of $\nu$
with those of $\mu.$
We have chosen $\nu$ so that
its support, namely the union of its positive and negative supports,
coincides with the support of $\mu;$ and since $\mu=D(\nu),$
the positive and negative supports of $\nu$ will each {\em contain}
the corresponding support of $\mu.$
So if these inclusions are not both equalities, we must have a vertex
$p$ which is both in the positive and the negative support of $\nu;$
i.e., such that there is an edge $(q,p)$ of
$\Gamma(\nu)$ leading into $p,$ and an edge $(p,r)$ leading out of it.
Let $\nu'=\nu_{q,r}-\nu_{q,p}-\nu_{p,r}.$
Like the element denoted by that symbol in the preceding argument,
this satisfies $D(\nu')=0.$
Let us again form $\nu+\lambda\,\nu',$ this time choosing the value
$\lambda>0$ which leads to the cancellation of the smaller
of the coefficients of $\nu_{q,p}$ and $\nu_{p,r}$ in $\nu,$
or of both if these coefficients are equal.
Since this does not reverse the sign of either
of these coefficients, $\nu+\lambda\,\nu'$ still belongs to $W_+.$
Note that $\int_{(x,y)\in X\times X} d(x,y)\,d(\nu{+}\lambda\nu')\leq
\int_{(x,y)\in X\times X} d(x,y)\,d\nu,$ since by the
triangle inequality, $d(q,r)\leq d(q,p)+d(p,r);$
so the property of minimizing the latter integral among elements
of $W_+$ mapped to $\mu$ by $D$ has not been lost.
Also, $\Gamma(\nu+\lambda\nu')$ has dropped at least one edge
that belonged to $\Gamma(\nu),$ since at least one coefficient
was canceled, and has gained at most one edge, namely $(q,r)$
(if that was not previously present); so the total number of
edges has not increased.
Finally, when we look at the sum of all the coefficients,
we see that the coefficient of $\nu_{q,r}$ has increased by $\lambda,$
while those of $\nu_{q,p}$ and $\nu_{p,r}$
have both decreased by $\lambda,$ so there has been a net
change of $-\lambda<0.$
Thus, we have a contradiction to our choice of $\nu$ as minimizing
that sum.
This completes the proof of the main assertion of the proposition.

Let us verify, finally, the parenthetical comment that
the $\nu$ of the proposition may not be unique.
Let $X$ be a $\!4\!$-point space $\{x_1,x_2,x_3,x_4\}$
where the distance between every pair of distinct points is $1,$
and let $\mu=\mu_{x_1}+\mu_{x_2}-\mu_{x_3}-\mu_{x_4}.$
It is not hard to check that in this case, the only elements $\nu$ that
can possibly satisfy the conditions of the proposition are
$\nu_{x_1,x_3}+\nu_{x_2,x_4}$ and $\nu_{x_1,x_4}+\nu_{x_2,x_3}.$
Since these give the same value for the integral of~(\ref{x.univnorm}),
$d(x_1,x_3)+d(x_2,x_4)=2=d(x_1,x_4)+d(x_2,x_3),$
each satisfies our conditions.

(Of course, for {\em most} choices of metric on this set
$X,$ one of these two values is smaller than the other, and we
then get a unique $\nu$ satisfying the conditions of the proposition.)
\end{proof}

To get a self-contained version of the above proof which includes
the existence result we cited from~\cite{NW}, one may start by
looking at any finite subset $X_0$ of $X$ containing the support
of $\mu,$ verify by compactness as above that
the infimum of~(\ref{x.univnorm}) over elements $\nu$ with support
contained in $X_0$ is achieved, then note that any element in the
support of $\nu$ but not in the support of $\mu$
must belong to both the positive and negative supports of $\nu,$
a situation excluded by the proof.
Letting $X_0$ then run over all finite subsets of $X$ containing
the support of $\mu,$ one sees that the infimum of~(\ref{x.univnorm})
exists, and is simply the infimum with
$\nu$ restricted to have support in the support of $\mu.$

We remark that the final condition of the above proposition, that
$\Gamma(\nu)$ have no cycles, is not entailed by the other conditions.
E.g., returning to $X=\{x_1,x_2,x_3,x_4\}$ with all distances $1,$
we see that every convex linear combination $\nu$ of the two elements
that we
found, $\nu_{x_1,x_3}+\nu_{x_2,x_4}$ and $\nu_{x_1,x_4}+\nu_{x_2,x_3},$
still minimizes~(\ref{x.univnorm}), and still has support $X;$ but
if $\nu$ is a proper convex combination of those two
elements, then $\Gamma(\nu)=\Gamma(\nu_{x_1,x_3}{+}\,\nu_{x_2,x_4})\cup
\Gamma(\nu_{x_1,x_4}{+}\,\nu_{x_2,x_3}),$ which contains (indeed, is)
a cycle.

Let us now show, however, that when, as in the statement of the
proposition, $\Gamma(\nu)$ is cycle-free, it uniquely determines $\nu.$
Thus, the calculation of the norm~(\ref{x.univnorm}) reduces in
principle to checking finitely many~$\nu.$

\begin{lemma}\label{L.*Gdetsnu}
Let $\mu\in U_0,$ and let $\Gamma$ be a directed graph
with vertex-set $X$ and without cycles.
Then there is at most one $\nu\in W$ \textup{(}and so, a fortiori,
at most one $\nu\in W_+)$ such that $D(\nu)=\mu$ and
$\Gamma(\nu)\subseteq\Gamma.$

To characterize this element $\nu,$ consider any edge $(x,y)$
in $\Gamma.$
Let $\Gamma_x$ \textup{(}containing $x)$ and $\Gamma_y$
\textup{(}containing $y)$ be the connected components into
which the connected component of
$\Gamma$ containing $(x,y)$ separates when that edge is removed.
Then the coefficient in $\nu$ of $\nu_{x,y}$ is the common
value of $\int_{\Gamma_x}d\mu$ and $-\int_{\Gamma_y}d\mu,$ i.e., is
both the sum of the coefficients of $\mu_z$ over $z$ in $\Gamma_x,$
and the negative of the corresponding sum over~$\Gamma_y.$
\end{lemma}\begin{proof}
We will prove the assertion of the second paragraph, from which
that of the first clearly follows.

Writing $\mu=D(\nu),$ the contributions to
the expression $\int_{\Gamma_x}d\mu$ from any
term $\nu_{p,q}$ such that both $p$ and $q$ lie in $\Gamma_x$ clearly
cancel, while terms such that neither $p$ nor $q$ lies in $\Gamma_x$
contribute nothing.
This leaves the $\nu_{x,y}$ term, which contributes precisely its
coefficient, leading to the first description of that coefficient.
Likewise, this term contributes the negative of its
coefficient to $\int_{\Gamma_y}d\mu,$ yielding the second description.
\end{proof}

\begin{corollary}\label{C.match}
Suppose $\mu\in U_0$ is integer-valued, and let
$\nu$ be an element of $W_+$ with the properties that $D(\nu)=\mu,$
that $\nu$ has the same positive and negative supports as $\mu,$
and that $\Gamma(\nu)$ has no cycles.

Then for every $x\in X$ such that the coefficient of $\mu_x$
in $\mu$ is $\pm 1,$ the vertex $x$ is a leaf of $\Gamma(\nu).$

Hence, if $\mu$ has the property that the coefficient of {\em every}
$\mu_x$ is $\pm 1,$ then $\nu$ is
induced, in the obvious way, by a bijection between the positive
support of $\mu$ and the negative support of $\mu.$
Thus, in that case, letting $n$ be the common cardinality of these
supports, there are exactly $n!$ such $\nu\in W_+.$
\end{corollary}\begin{proof}
Consider any $x$ such that $\mu_x$ has coefficient $+1$ in $\mu.$
Then $x$ is in the positive support
of $\Gamma(\nu),$ but not in the negative support.
The latter condition says that $\nu$
involves no terms $\nu_{y,x}$ $(y\in X),$
so $+1$ is the sum of the coefficients in $\nu$
of the terms $\nu_{x,y}$ $(y\in X-\{x\}).$
Since $\nu\in W_+,$ these coefficients are nonnegative, and
by Lemma~\ref{L.*Gdetsnu} they are integers, so as they sum
to $1,$ only one of them can be nonzero, making $x$ a leaf.
The same argument, mutatis mutandis, gives the case
where the coefficient of $\mu_x$ is~$-1.$

The assertion of the final paragraph clearly follows, since
a directed graph in which every vertex is a leaf
corresponds to a bijection between ``source'' and ``sink'' vertices.
\end{proof}

As sample applications, recall the two computations at the
beginning of this section, with which we
motivated the construction of our universal embedding
$e:X\rightarrow U_0.$
In our present notation, what we were doing
was evaluating the norms in $U_0$ of elements of
the two forms $\mu_{x_1}+\mu_{x_2}-\mu_{x_3}-\mu_{x_4}$ and
$3\mu_{x_1}+\mu_{x_2}-2\mu_{x_3}-2\mu_{x_4}.$
In the first case, the last paragraph of the above corollary leads
to just two graphs, and hence two values of $\nu$ one of which
must achieve the
infimum~(\ref{x.univnorm}), namely $\nu_{x_1,x_3}+\nu_{x_2,x_4}$
and $\nu_{x_1,x_4}+\nu_{x_2,x_3},$ showing that if $f$ is our universal
map $e,$ equality holds in~(\ref{x.x1234}), and
for general $f,$~(\ref{x.x1234}) is the best bound.
(This also establishes the example that we said was ``not hard
to check'' in the next-to-last paragraph of the proof of
Proposition~\ref{P.graph}.)

In the case $\mu=3\mu_{x_1}+\mu_{x_2}-2\mu_{x_3}-2\mu_{x_4},$
Corollary~\ref{C.match} says that
$x_2$ is a leaf of $\Gamma(\nu).$
As it lies in the positive support of $\mu,$ the vertex it is attached
to must lie in the negative support, i.e., must
be either $x_3$ or $x_4.$
In the former case, subtracting $\nu_{x_2,x_3}$ from $\nu$
will give an element $\nu'\in W_+$ which is sent by $D$
to $3\mu_{x_1}-\mu_{x_3}-2\mu_{x_4}.$
Since this $\nu'$ has only one element, $x_1,$
in its positive support, its graph
is uniquely determined, giving $\nu'=\nu_{x_1,x_3}+2\nu_{x_1,x_4},$
hence $\nu=\nu_{x_2,x_3}+\nu_{x_1,x_3}+2\nu_{x_1,x_4}.$
The case where $x_2$ is attached to $x_4$ similarly gives
$\nu=\nu_{x_2,x_4}+2\nu_{x_1,x_3}+\nu_{x_1,x_4},$ and these
together show that~(\ref{x.x1234'}) is a best bound.

I referred earlier to the mathematical
economist's ``transportation problem''.
There, our $d(x,y)$ corresponds to
the cost of transporting a unit quantity of goods from location $x$
to location $y;$ so our definition of $||\mu||$ describes
the minimum cost of transporting goods produced and consumed
at locations and in quantities specified by $\mu.$\vspace{6pt}

Incidentally, the first assertion of Corollary~\ref{C.match} does not
remain true if we weaken ``the coefficient of $\mu_x$ in $\mu$ is
$\pm 1\!$'' to ``the coefficient of $\mu_x$ has least
absolute value among the nonzero coefficients occurring in $\mu.\!$''
For instance, suppose $\mu$ has the form
$3\mu_{x_1}-4\mu_{x_2}+2\mu_{x_3}-4\mu_{x_4}+3\mu_{x_5}.$
Then one of the elements of $W_+$
satisfying the conditions of Corollary~\ref{C.match} is
$\nu=3\nu_{x_1,x_2}+\nu_{x_3,x_2}+\nu_{x_3,x_4}+3\nu_{x_5,x_4}.$
Here $\Gamma(\nu)$ has the form
$x_1\rightarrow x_2\leftarrow x_3\rightarrow x_4\leftarrow x_5,$
so $x_3,$ despite having smallest coefficient in $\mu,$ is not a leaf.
\vspace{6pt}

Proposition~\ref{P.graph} sheds some light on our earlier
``partial universality'' result, Corollary~\ref{C.notuniv}(i).
Given any {\em convex} linear combination of points of our universal
image of $X,$ $\sum_I a_i\<\mu_{x_i}$ $(a_i>0,\ \sum a_i=1),$ and any
point $x\in X$ (which for simplicity we will assume is
not one of the $x_i,$ though the argument can be adjusted to the
case where it is), the difference $\mu_x-\sum a_i\<\mu_{x_i}$ is
an element of $U_0$ with positive support $\{x\}$ and
negative support $\{x_i\mid i\in I\}.$
For this situation, the conditions of Proposition~\ref{P.graph} clearly
lead to a unique $\Gamma(\nu),$ to wit, the tree whose edges
are all pairs $(x,x_i)$ $(i\in I),$ and hence to the unique
choice $\nu=\sum_i a_i\,\nu_{x,x_i}.$
Thus, the right-hand
side of~(\ref{x.univnorm}) comes to $\sum_i a_i\,d(x,x_i),$
which is equal to the right-hand side of~(\ref{x.g_<e}).
In contrast, when one considers the difference between
two general convex linear combinations of elements $\mu_x,$
as in Corollary~\ref{C.notuniv}(ii),
there may be many directed graphs satisfying the conditions
of Proposition~\ref{P.graph}, so the norm of that difference
doesn't have a simple expression.

The universality of the Arens-Eells space $U$ yields
a formula for $\mr(X,\fb{NmV})$ analogous to our
description~(\ref{x.mrConv}) of $\mr(X,\fb{Conv});$ namely,
\begin{xlist}\item\label{x.mrNmV}
$\mr(X,\fb{NmV})\ =\ \inf_{\mu\in U_1}\ \sup_{x\in X}
\ \inf_{\nu\in W_+,\ D(\nu)=\mu-\mu_x}
\ \int_{y,z\in X}\,d(y,z)\ d\nu(y,z).$
\end{xlist}
But this is cumbersome to use.
E.g., the reader might try working through a verification,
for the space described by~(\ref{x.xyz}),
of the statement that the $\mu$ implicit in~(\ref{x.p=3/2-1/2}),
$(\mu_{y_0}+\mu_{y_1}+\mu_{y_2}-\mu_{x})/2,$
does indeed lead to the infimum of~(\ref{x.mrNmV}),
showing that $\mr(X,\fb{NmV})=5/2,$ and not a smaller value.

Given an element $\mu\in U_0,$ it would be interesting to look for
bounds on the number of distinct graphs $\Gamma(\nu)$
corresponding to elements $\nu\in W_+$ as in the first sentence of
Corollary~\ref{C.match}.
(This is simply a function of the coefficients occurring
in $\mu,$ as a family of positive real numbers with multiplicities.)
To start with, one might look for bounds in terms of the
cardinalities of the positive and negative supports of~$\mu.$
\vspace{6pt}

Weaver~\cite{NW} also gets a description of the universal nonexpanding
map of $X$ into a normed {\em complex} vector space, paralleling
the description for the real case, but he
notes~\cite[p.43, next-to-last paragraph of \S2.2]{NW} that in the
complex case it is no longer true that the infimum
corresponding to~(\ref{x.univnorm}) is always attained by a $\nu$
having the same support as $\mu.$
(In the complex version of~(\ref{x.univnorm}), by the way, one must
replace $d\nu$ by $|d\nu|,$ instead of restricting $\nu$ to a
``positive cone''
as above, since there is no natural analog of that cone.
Weaver takes this approach for both the real and complex cases;
my use of $W_+$ for the real case is one of the different
technical choices that I have made.)
For instance, if $X=\{x,y_0,y_1,y_2\}$ with
$d(x,y_i)=1$ and $d(y_i,y_j)=2$ $(i\neq j),$ and
if $\mu=\mu_{y_0}+\omega\<\mu_{y_1}+\omega^2\mu_{y_2},$ where $\omega$
is a primitive cube root of unity, then the minimizing $\nu$ is
$\nu_{y_0,x}+\omega\<\nu_{y_1,x}+\omega^2\nu_{y_2,x},$ which makes
that integral $3,$ while the best $\nu$ having support in the
support of $\mu,$ $\{y_0,y_1,y_2\},$ is
$\frac{1}{3}(1-\omega)\nu_{y_0,y_1}+
\frac{1}{3}(\omega-\omega^2)\nu_{y_1,y_2}+
\frac{1}{3}(\omega^2-1)\nu_{y_2,y_0},$
of which each term contributes $\frac{1}{3}\cdot\sqrt{3}\cdot 2$
to that integral, giving a total of $\sqrt{3}\cdot 2=\sqrt{12}>3.$
If in this space we replace the point $x$ by a sequence of
points $x_1,x_2,\dots,$ such that $d(x_m,y_i)=1+1/m$ and
$d(x_m,x_n)=|1/m-1/n|,$ the above $\mu$ still has
$||\mu||=3,$ but the infimum defining that norm is not achieved.

\section{Appendix: Translating convex sets to $0.$}\label{S.park}
In \S\ref{S.discr}, we saw that the radius of a subset $X$ of a
normed vector space $V$ could be larger when measured within a convex
subset $C$ of $V$ than within the whole space $V.$
If we regard this as a pathology, we would like to know in
which $V$ it does not occur.
We shall obtain partial results below, which, we will see, make
it likely that for $n>2,$ the only norms on $\R^n$ for which it
does not happen are those giving a structure isomorphic to $\E^n.$

Observe that the radius of $X,$ whether within $V$ or within a convex
subset $C,$ is determined by the set of closed balls containing $X,$
and that these are all convex; hence that
radius is a function of the convex hull of $X.$
So our question reduces to the case where $X$ is convex.
Moreover, if $X$ shows the above behavior with respect to one convex
subset $C$ of $V,$ it will show it with respect to any smaller convex
subset in which it lies; these two observations reduce
our question to the case where $C=X.$
This reduction is the equivalence of conditions~(\ref{x.rad_VX=rad_CX})
and~(\ref{x.rad_V=rad_C}) of the next lemma.
Condition~(\ref{x.bring_to_0}) then reformulates the problem.

(Note that in~(\ref{x.bring_to_0}) and similar
statements throughout this section, an expression such as
``$\!C-v\!$'' will denote the translate of the set $C$ by the vector
$-v,$ in contrast to notations such as $X-\{x\}$ for set-theoretic
difference, used occasionally in earlier sections.)

\begin{lemma}\label{L.external}
If $V$ is a locally compact normed vector space,
with closed unit ball $B,$
then the following conditions are equivalent:
\begin{xlist}\item\label{x.rad_VX=rad_CX}
For every nonempty subset $X$ of $V,$ and convex subset
$C$ of $V$ containing $X,$ one has $\rad_V(X)=\rad_C(X).$
\end{xlist}
\begin{xlist}\item\label{x.rad_V=rad_C}
For every nonempty
convex subset $C$ of $V,$ one has $\rad_V(C)=\rad_C(C).$
\end{xlist}
\begin{xlist}\item\label{x.bring_to_0}
Every nonempty closed convex subset $C$ of $B$ has a translate $C-v$
which contains $0$ and is again contained in $B.$
\end{xlist}
\end{lemma}\begin{proof}
We have noted the equivalence of~(\ref{x.rad_VX=rad_CX})
and~(\ref{x.rad_V=rad_C}); let us prove~(\ref{x.rad_V=rad_C})
equivalent to~(\ref{x.bring_to_0}).

(\ref{x.bring_to_0})$\Rightarrow$(\ref{x.rad_V=rad_C}):
Dilating by arbitrary constants, we see
that if~(\ref{x.bring_to_0}) holds for $B,$ then
it holds for $rB$ for all positive real numbers $r.$
Moreover, the statement that $C-v$ contains $0$ and is contained in $rB$
is equivalent to saying that $v\in C$ and that $v+rB$ contains $C;$
i.e., that $C$ is contained in the ball of radius $r$ about $v\in C.$
Thus~(\ref{x.bring_to_0}) says that if a closed convex set $C$ is
contained in some closed ball about some point of $V$
(taken without loss of generality to be $0),$ then it is contained
in a ball of the same radius about one of its own points.
This yields the case of~(\ref{x.rad_V=rad_C}) where $C$ is closed.
The facts that the convex hull of a finite subset of $V$ is compact,
hence closed, and that the radius of an arbitrary set is the supremum
of the radii of its finite subsets, allow us to deduce the general
case of~(\ref{x.rad_V=rad_C}) from the case of closed~$C.$

(\ref{x.rad_V=rad_C})$\Rightarrow$(\ref{x.bring_to_0}):
If $C$ is a closed convex subset of $V$ contained in $B,$ then
$\rad_V(C)\leq 1,$ so by~(\ref{x.rad_V=rad_C}), $\rad_C(C)\leq 1.$
Moreover, compactness of $C$ implies that
the set of radii of closed balls containing $C$ and centered
at points $v\in C$ achieves this minimum $\rad_C(C)\leq 1,$
so that $C$ is contained in a translate $v+B$ $(v\in C),$
i.e., $C-v\subseteq B.$
\end{proof}

Now~(\ref{x.bring_to_0}) is a statement purely about
the convex set $B$ in the topological vector
space $V,$ so our question becomes that of which subsets $B$
of a topological vector space $V$ satisfy it.
(In the statement of the lemma, the topology and the
set $B$ both arise from the normed structure on $V;$ but that
relation is not needed by the statement of~(\ref{x.bring_to_0}) alone.)
Here are some pieces of language, one ad hoc, the rest more or less
familiar, that we will use in examining this question.

\begin{definition}\label{D.park}
If $C\subseteq B$ are convex subsets of a real vector space $V,$
with $0\in B,$ we shall call $C$ {\em parkable} in $B$ if there
exists $v\in C$ such that $C-v\subseteq B.$
When clear from context, ``in $B\!$'' may be omitted.

If $V$ is a real topological vector space, we will call sets
of the form $\{x\in V\mid L(x)=a\},$ where $L$ is a nonzero continuous
linear functional on $V$ and $a\in\R,$ {\em hyperplanes,} while sets
of the form $\{x\in V\mid L(x)\geq a\}$ will be called
{\em closed half-spaces.}

A subset $S$ of a vector space $V$ will be called
{\em centrally symmetric} if $S=-S.$
A {\em center of symmetry} of a subset $S$ of $V$ will mean a point
$v\in V$ such that $S-v$ is centrally symmetric;
equivalently, such that $S=2v-S.$
\textup{(}Note that a center of symmetry of a nonempty
convex set belongs to that set.\textup{)}
\end{definition}

\begin{lemma}\label{L.park}
Let $B$ be a compact convex subset of $\R^n$ containing $0.$
Then the following conditions are equivalent:
\begin{xlist}\item\label{x.parkAB}
The intersection of $B$ with every hyperplane $A$ that
meets $B$ is parkable.
\end{xlist}
\begin{xlist}\item\label{x.parkHB}
The intersection of $B$ with every closed half-space
$H$ that meets $B$ is parkable.
\end{xlist}
\begin{xlist}\item\label{x.parkC}
Every nonempty closed convex subset $C$ of $B$ is parkable
$(=$\textup{(\ref{x.bring_to_0})} above\textup{)}.
\end{xlist}
\end{lemma}\begin{proof}
(\ref{x.parkC})$\Rightarrow$(\ref{x.parkAB}) is clear; we will show
(\ref{x.parkAB})$\Rightarrow$(\ref{x.parkHB})$\Rightarrow$%
(\ref{x.parkC}).

(\ref{x.parkAB})$\Rightarrow$(\ref{x.parkHB}):
Let $H$ be a closed half-space in $\R^n,$ bounded by a hyperplane $A,$
and meeting $B.$
If $H\cap B$ contains $0$ it is trivially parkable, so assume the
contrary.
Thus $B$ meets both $H$ and its complement, hence it meets their
common boundary $A,$ so by~(\ref{x.parkAB}) there exists
$v\in A\cap B$ such that $(A\cap B)-v\subseteq B.$
I claim that $(H\cap B)-v$ is also contained in $B.$
Indeed, let $p\in H\cap B;$ we wish to show $p-v\in B.$
Intersecting our sets with the subspace of $V$ spanned by $p$ and $v,$
and taking appropriate coordinates, we may assume that
$n=2,$ that $A$ is the line $\{(x,y)\mid y=1\}\subseteq\R^2,$ and
that $v$ is the point $(0,1).$
$H$ will be the closed half-plane $\{(x,y)\mid y\geq 1\},$
so we can write $p=(x_p,y_p)$ with $y_p \geq 1.$

In this situation, $A\cap B$ will be a line segment
(possibly degenerate) extending from a
point $(s,1)$ to a point $(t,1)$ $(s\leq t).$
Since $(A\cap B)-v\subseteq B,$ $B$ also contains
the segment from $(s,0)$ to $(t,0).$
Note that if $x_p$ were $>t,$ then the point where
the line segment from $p=(x_p,y_p)\in B$ to $(t,0)\in B$ meets
$A$ would have $\!x\!$-coordinate $>t,$ contradicting the assumption
that $A\cap B$ terminates on the right at $(t,1);$ so $x_p\leq t.$
Similarly, $x_p\geq s.$
Thus, $x_p\in [s,t],$ so $(x_p,0)\in B.$
Hence $p-v=(x_p,\,y_p{-}1)$ lies on the line segment connecting
$p=(x_p,y_p)\in B$ with $(x_p,0)\in B,$ hence lies in $B,$ as claimed.

(\ref{x.parkHB})$\Rightarrow$(\ref{x.parkC}):
Suppose $C$ is a nonempty closed convex subset of $B$ which
is not parkable.
By compactness of $B,$ among the translates of
$C$ contained in $B$ there is (at least) one that minimizes its
distance to $0$ in the Euclidean norm on $\R^n;$
let us assume $C$ itself has this property.
Let $p$ be the point of $C$ nearest to $0$ in that norm, and let
$A$ be the hyperplane passing through $p$ and perpendicular (again
in the Euclidean norm) to $p$ regarded as a vector.
Then $C$ will lie wholly in the half-space
$H$ bounded by $A$ and not containing~$0.$
(For if we had $q\in C$ not lying in $H,$
then points close to $p$ on the line segment
from $p$ to $q$ would be nearer to $0$ than $p$ is.)
Assuming~(\ref{x.parkHB}), $H\cap B$ is parkable;
say $v\in H\cap B$ with $(H\cap B)-v\subseteq B.$
Since $v\in H,$ if we write $v$ as the sum $a\<p+q$ of a scalar
multiple of $p$ and a vector $q$ perpendicular to $p,$ the
coefficient $a$ will be $\geq 1,$ and so in particular, positive.
It follows that for sufficiently small positive $c,$ the point $p-c\<v$
will be closer to $0$ than $p$ is; moreover, if we take such
a $c$ that is $\leq 1,$ $(H\cap B)-c\<v$ will still be contained in $B,$
since $H\cap B$ and $(H\cap B)-v$ are.
Hence $C-c\<v$ is contained in $B,$ and has a point $p-c\<v$ which is
closer to $0$ than $p$ is, contradicting our minimality assumption
on $C$ and~$p.$
\end{proof}

(In the above result, we could have replaced $\R^n$ by any
real Hilbert space.)

Clearly, the closed Euclidean unit ball in $\R^n$
satisfies~(\ref{x.parkAB}), and hence~(\ref{x.parkHB})
and~(\ref{x.parkC}); hence since those conditions are preserved by
invertible linear transformations, so does the closed
region enclosed by any ellipsoid centered at $0.$
On the other hand, our example in the paragraph
containing~(\ref{x.011}), of a normed vector space in
which~(\ref{x.rad_VX=rad_CX}) failed, had for its unit ball $B$ a
$\!3\!$-cube centered at $0,$ showing that our properties
fail for that $B.$
To see geometrically the failure of~(\ref{x.parkAB}) for that
$B,$ choose a vertex of that cube and pass a
plane $A$ through the three vertices adjacent thereto;
it is not hard to see that $A\cap B$ is not parkable.
One can similarly show that none of the regular polyhedra centered
at $0$ satisfy~(\ref{x.parkAB}),
nor a circular cylinder centered at $0,$
% pass A through a tangent to top circle, parallel chord
% near antipodal point of bottom circle
nor the solid obtained by attaching a hemisphere to the
top and bottom of such a cylinder.
In fact, for $n>2,$ I know of no compact convex subset of $\R^n$ with
nonempty interior that does satisfy that condition, other than the
regions enclosed by ellipsoids centered at $0.$
The situation is different for $n=2,$ as shown by
point~(d) of the next result.

\begin{lemma}\label{L.csym}
Suppose $B$ is a {\em centrally symmetric} convex subset of $\R^n.$
Then\\[3pt]
\textup{(a)}~ Any nonempty convex subset $C\subseteq B$ that has a
center of symmetry is parkable in $B.$\vspace{3pt}

Hence, assuming in the remaining points that $B$ is also {\em compact,}
we have\\[3pt]
\textup{(b)}~ If the intersection of $B$ with every hyperplane $A$ that
meets $B$ has a center of symmetry, then $B$ satisfies the equivalent
conditions~\textup{(\ref{x.parkAB})-(\ref{x.parkC})}.\vspace{3pt}

In particular,\\[3pt]
\textup{(c)}~ If $B$ is the closed region enclosed by an
ellipsoid in $\R^n,$ then $B$
satisfies~\textup{(\ref{x.parkAB})-(\ref{x.parkC})}, and\\[3pt]
\textup{(d)}~ If $n=2,$ then without further restrictions, $B$
satisfies~\textup{(\ref{x.parkAB})-(\ref{x.parkC})}.
\end{lemma}\begin{proof}
Let $C$ be as in~(a), with center of symmetry $z\in C.$
Then for every $x\in C,$ $2z-x\in C\subseteq B,$
so by central symmetry of $B,$ we have $x-2z\in B.$
Averaging $x$ and $x-2z,$ we get $x-z\in B.$
Thus $C-z\subseteq B,$ so $C$ is parkable.

It follows that any $B$ as in~(b)
satisfies~(\ref{x.parkAB}), hence by~Lemma~\ref{L.park},
all of (\ref{x.parkAB})-(\ref{x.parkC}).

In the situation of~(c), the intersection of $B$ with a
hyperplane $A,$ if nonempty, is either a point or the region
enclosed by an ellipsoid in $A,$ hence has a center of symmetry,
while in~(d) the intersection of $B$ with every line in
$\R^2$ that meets $B$ is a point or a closed line segment, hence has a
center of symmetry; so in each case,~(b) gives the asserted conclusion.
\end{proof}

\begin{question}\label{Q.park}
Let $n\geq 3,$ and suppose $B$ is a compact convex subset
of $\R^n$ having nonempty interior and containing $0.$
Of the implications \textup{(i)$\Rightarrow$(ii)$\Rightarrow$(iii)},
which we have noted hold among the conditions listed below,
is either or both reversible?\\[6pt]
\textup{(i)} $B$ is an ellipsoid centered at $0.$\\[6pt]
\textup{(ii)} $B$ is centrally symmetric, and for every hyperplane
$A$ meeting $B,$ $A\cap B$ has a center of symmetry.\\[6pt]
\textup{(iii)} ~Every closed convex subset of $B$ is parkable in $B.$
\end{question}

Branko Gr\"unbaum has pointed out to me a similarity between this
question and the result of W.\,Blaschke \cite[pp.157--159]{B} that if
$E$ is a smooth compact convex surface in $\R^3$ with everywhere
nonzero Gaussian curvature, such that when $E$
is illuminated by parallel rays from any direction, the boundary
curve of the bright side lies in a plane, then $E$ is an ellipsoid.
I believe that methods similar to Blaschke's may indeed show that
both implications of Question~\ref{Q.park} are reversible.
To see why, suppose $B$ is a compact convex subset of
$\R^3$ with nonempty interior containing $0,$
which satisfies~(iii) above, and whose boundary $E$ is (as in Blaschke's
result) a smooth surface with everywhere nonzero Gaussian curvature.
Let $A$ be any plane through $0,$ and $A'$ the plane gotten by
shifting $A$ a small distance.
Now the vectors that can possibly park $A'\cap B$ are
constrained by the directions of the tangent planes to $E$ at the points
of $A'\cap E$ (which are well-defined because $E$ is assumed smooth),
and if we take $A'$ sufficiently close to $A,$ these tangent planes
become close to the corresponding
tangent planes at the points of $A\cap E.$
Applying the above observations to planes $A'$ on both sides of
$A,$ one can deduce that all the tangent planes
to $E$ along $A\cap E$ must contain vectors in some common
direction (I am grateful to Bjorn Poonen for this precise
formulation of a rough idea I showed him);
in other words, that $A\cap E$ is the boundary of the
bright side when $E$ is illuminated by parallel
rays from that direction.
By definition, $A\cap E$ lies in the plane $A;$ so we have
the situation that Blaschke considered, except that we have started with
planarity and concluded that the curve is a boundary of illumination,
rather than vice versa.

That last difference is probably not too hard to overcome.
More serious is the smoothness assumption on $E\<,$ used in
both the above discussion and Blaschke's argument.
Finally, can the result be pushed from $n=3$ to arbitrary $n\geq 3$?
I leave it to those more skilled than I in the subject to see whether
these ideas can indeed be turned into a proof that
(iii)$\!\Rightarrow\!$(i) in Question~\ref{Q.park}.

A related argument which can be extracted from a step in Blaschke's
development shows that a compact convex subset $B$ of $\R^2$
containing $0$ and satisfying (\ref{x.parkC}), whose boundary is a
smooth curve containing no line segments, must be centrally symmetric.
Again, one would hope to remove the conditions on the boundary.

One can ask about a converse to another of our observations:
\begin{question}\label{Q.park_C}
Suppose $C$ is a compact convex subset of $\R^n$ $(n>2)$ such that
for every centrally symmetric compact convex subset $B$ of $\R^n$
containing a translate $C'$ of $C,$ the set $C'$ is parkable in $B.$
Must $C$ have a center of symmetry?
\end{question}

Here the behavior of a given $C$ can change depending on
whether the dimension of the ambient vector space is $2$ or
-- as in the above question -- larger: a triangle $C$ has the above
property in $\R^2$ by Lemma~\ref{L.csym}(d), but not in $\R^3,$ as
we saw in the example where $B$ was a cube.\vspace{6pt}

Returning to the ``pathology'' which motivated the considerations of
this section, one important case is where the radius of a
subset $X$ of a normed vector space $V$ decreases
when $V$ is embedded in a larger normed vector space $W.$
The next lemma determines how far down the radius of a given $X$ can go.

\begin{lemma}\label{L.enlarge}
Let $V$ be a normed vector space, and $X$ a bounded subset of $V.$
Then
\begin{xlist}\item\label{x.enlarge}
$\inf_{W\supseteq V}\ \rad_W(X)\ =\ \rad_V \{(x-y)/2\mid x,\,y\in X\},$
\end{xlist}
where $W$ ranges over all normed vector spaces containing $V.$
This infimum is realized by a $W$ in which $V$ has codimension $1.$
\end{lemma}

\begin{proof}
First consider any normed vector space $W$ containing $V,$ and
suppose $X$ is contained in the closed ball of radius $r$
about $w\in W.$
That ball has $w$ as a center of symmetry, so it also
contains $\{2w-y\mid y\in X\},$ hence taking midpoints of
segments connecting that set to points $x\in X,$ it contains
$\{w+(x-y)/2\mid x,y\in X\}.$
Translating by $-w,$ we see that the ball of radius $r$
about $0$ contains $\{(x-y)/2\mid x,y\in X\},$
so $r$ is at least the right-hand side of~(\ref{x.enlarge}).
This gives the inequality ``$\geq$'' in~(\ref{x.enlarge});
it remains to construct a $W$ for which
$\rad_W(X)$ equals that right-hand side.

Before doing this, note that~(\ref{x.enlarge})
holds trivially if $X$ is empty or a singleton; so assuming
it is neither, let us re-scale and assume without loss of generality
that the right-hand side of~(\ref{x.enlarge}) equals $1.$
Since the set whose radius is taken
there is centrally symmetric, that set is contained in the
closed unit ball $B_V$ of $V.$
(Cf.\ the proof of Lemma~\ref{L.csym}(a), which works not just
for $\R^n,$ but for any normed vector space with $B$ its closed unit
ball; or the proof of Lemma~\ref{L.sym}, applied to the $\!2\!$-element
group generated by $x\mapsto -x.)$
Now let $W=V\oplus\R,$ let us identify $V$ with
$V\times\{0\}\subseteq W,$ and let us take for the closed unit ball
$B_W$ of $W$ the closure of the convex hull of
\begin{xlist}\item\label{x.3parts}
$\{(x,1)\mid x\in X\}\ \cup\ B_V\ \cup \ \{(-x,-1)\mid x\in X\}.$
\end{xlist}
(We understand ``closure'' to mean ``with respect to the
product topology'', since we don't have a norm until
we have made the above definition.)
It is easy to see that any point in the convex hull
of~(\ref{x.3parts}) whose second coordinate is $0$
is a convex linear combination of a point of
$B_V$ and a member of the set on the right-hand
side of~(\ref{x.enlarge}); but by assumption that set is
contained in $B_V;$ so in fact, $B_W\cap V=B_V,$
so the norm of $W$ indeed extends that of $V.$

But $B_W$ contains the translate $\{(x,1)\mid x\in X\}$
of $X,$ hence $X$ is contained in the closed ball of
radius $1$ about $(0,-1),$ hence has radius $\leq 1$ in $W.$
\end{proof}

Even the case $V=\E^n$ is not immune to this phenomenon,
since even in that case, the overspace $W$ of the above
construction is generally not Euclidean.
For instance, if we take for $X$ an equilateral triangle in $\E^2$
centered at the origin, it is not hard to see that
$\{(x-y)/2\mid x,y\in X\}$ is a hexagon whose vertices are
the midpoints of the edges a regular hexagon with the same circumcircle
as $X;$ so the radius of $X$ decreases in $W$ by the ratio of the
inradius to the circumradius of a regular hexagon,
in other words, by $\sqrt 3/2.$

This will not, of course, happen for a centrally symmetric $X$
(cf.\ Lemma~\ref{L.csym} or~(\ref{x.enlarge})).
Other cases for which it cannot happen depend on the metric
structure: if $X$ is a right or obtuse triangle in $\E^2,$
or more generally, any bounded set containing a diameter of
a closed ball in which it lies,
its radius clearly cannot go down under extension of the
ambient normed vector space (cf.\ Corollary~\ref{C.r=d/2}).

\section{Acknowledgements.}\label{S.ackn}
In addition to persons acknowledged above,
I am indebted to W.\,Kahan for showing me an exercise he had given
his Putnam-preparation class, of proving Lemma~\ref{L.S}
in $\E^3,$ and for subsequently pointing out that my solution to that
exercise worked in any normed vector space; to Nik Weaver for
information about his results in~\cite{NW}, and to David Gale
for pointing out the connection between the construction
of \S\ref{S.AE} and results in mathematical economics.

% - - - - - - - - - - - - - - - - - - - - - - - - - - - - - -

George M. Bergman\\
Department of Mathematics\\
University of California\\
Berkeley, CA 94720-3840\\
USA\\
gbergman@math.berkeley.edu
\end{document}